\documentclass{amsart}
\usepackage{graphicx} 
\usepackage{amsmath,amsfonts,euscript,amscd,amsthm,amssymb,upref,graphics,color,verbatim, float, placeins,mathrsfs, mathtools,comment,tikz, multirow, makecell}

\newcommand{\ba}{\textbf{a}}
\newcommand{\bb}{\textbf{b}}
\usepackage{tikz-cd}

\usepackage{longtable,array}

\usepackage{ulem}  

\usepackage[margin=1in]{geometry}
\usepackage[all]{xy}
\usepackage{hyperref}
\usepackage[shortlabels]{enumitem}

{\begin{enumerate}\setlength{\itemsep}{#1}}{\end{enumerate}}

\newcommand{\lnd}{\operatorname{{\rm LND}}}

\newcommand{\Integ}{\ensuremath{\mathbb{Z}}}
\newcommand{\Nat}{\ensuremath{\mathbb{N}}}
\newcommand{\Rat}{\ensuremath{\mathbb{Q}}}
\newcommand{\Comp}{\ensuremath{\mathbb{C}}}
\newcommand{\Reals}
{\ensuremath{\mathbb{R}}}

\newcommand{\aff}{\ensuremath{\mathbb{A}}}
\newcommand{\bk}{{\ensuremath{\rm \bf k}}}

\newcommand{\lb}{\langle}
\newcommand{\rb}{\rangle}
\newcommand{\trdeg}{	\operatorname{{\rm trdeg}}}
\newcommand{\Frac}{		\operatorname{{\rm Frac}}}

\newcommand{\lcm}{		\operatorname{{\rm lcm}}}
\newcommand{\Pic}{		\operatorname{{\rm Pic}}}

\newcommand{\cotype}{		\operatorname{{\rm cotype}}}
\newcommand{\setspec}[2]{\big\{\,#1\, \mid \,#2\, \big\}}
\newcommand{\isom}{\cong}
\newcommand{\Sing}{		\operatorname{{\rm Sing}}}
\newcommand{\PPP}{\mathbb{P}}

\newcommand{\height}{		\operatorname{{\rm ht}}}

\newcommand{\Spec}{		\operatorname{{\rm Spec}}}
\newcommand{\Proj}{		\operatorname{{\rm Proj}}}

\newcommand{\Supp}{		\operatorname{{\rm Supp}}}

\newcommand{\OSheaf}{\operatorname{\mathcal O}}

\newtheorem{theorem}[subsection]{Theorem}
\newtheorem*{theorem*}{Theorem}
\newtheorem{proposition}[subsection]{Proposition}
\newtheorem*{proposition*}{Proposition}
\newtheorem{lemma}[subsection]{Lemma}
\newtheorem*{lemma*}{Lemma}
\newtheorem{corollary}[subsection]{Corollary}
\newtheorem{conjecture}[subsection]{Conjecture}

\theoremstyle{definition}

\newtheorem{remark}[subsection]{Remark}

\newtheorem{definition}[subsection]{Definition}
\newtheorem{definitions}[subsection]{Definitions}

\newtheorem{nothing}[subsection]{}
\newtheorem{nothing*}[subsection]{}
\newtheorem{example}[subsection]{Example}
\newtheorem{question}[subsection]{Question}
\newtheorem{problem}[subsection]{Problem}
\newtheorem{assumption}[subsection]{Assumption}
\newtheorem{notation}[subsection]{Notation}
\newtheorem{remarks}[subsection]{Remarks}

\newcommand{\codim}{		\operatorname{{\rm codim}}}

\newcommand{\Id}{		\operatorname{{\rm Id}}}

\newcommand{\pgoth}{\mathfrak{p}}

\newcommand{\qgoth}{\mathfrak{q}}
\newcommand{\ggoth}{\mathfrak{g}}

\newcommand{\mgoth}{\mathfrak{m}}

\newcommand{\Div}{		\operatorname{{\rm Div}}}

\renewcommand{\div}{	\operatorname{{\rm div}}}
\newcommand{\Cl}{		\operatorname{{\rm Cl}}}
\newcommand{\depth}{		\operatorname{{\rm depth}}}

\newcommand{\CaCl}{		\operatorname{{\rm CaCl}}}

\newcommand{\Aff}{\mathbb{A}}

\title{A Survey on Pham-Brieskorn Varieties}
\author{Michael Chitayat}
\date{September 17th 2026}

\address{Dipartimento di Matematica ``Tullio Levi-Civita'', 
Universit\`a di Padova, Via Trieste 63, I-35121 Padova}
\email{michael.chitayat@unipd.it, mikechitayat@gmail.com}

\begin{document}

\maketitle

\section{Introduction}
Throughout my Ph.D. thesis, I studied a conjecture posed by Kaliman and Zaidenberg on the rigidity of a certain class of rings, called Pham-Brieskorn rings. While working on that conjecture, I learned many things about this family of rings and their corresponding varieties, mostly as applications of other theoretical results. Here, I collect some of what I learned so that others who are interested in this family of rings (or varieties) can find useful information in a single place. 

I try to make this survey self-contained by giving definitions and references for anything that can't be found in a standard textbook. The algebraic (resp. geometric) statements should be accessible to a student with a graduate course in commutative algebra (resp. algebraic geometry). Note that most of what is described in this survey is valid much more generally than in the special cases considered here. The goal is simply to expose the reader to some algebraic and geometric ideas using the family of Pham-Brieskorn rings as special examples.  

Note that there are not so many complete proofs in these notes and many of the given proofs strongly depend on other foundational results. The proofs that I do include will generally appear for one of the following reasons:
\begin{itemize}
\item I couldn't find the suitable statement in the literature in the form that I wanted,
\item I offer a (generally minor) generalization of the statement in the literature,
\item I want to demonstrate an idea or technique,
\item I want to write something down that is usually treated as folklore. 
\end{itemize}
Hopefully the text will be useful for graduate students as they explore some of these topics in their own coursework and research. 
\\

\noindent \textbf{Use of artificial intelligence.} I made use of Claude Opus 5 to find errors in, review, clarify and improve the text, to produce various diagrams and to produce the tuples in Section \ref{threefoldList}. I have checked the final version, and I believe it to be correct.

\section{Notation and Conventions}
Throughout this article we assume the following notation and conventions:

\begin{itemize}
    \item The set of natural numbers is denoted by $\Nat$ and contains $0$. The set of positive integers is denoted by $\Nat^+$ and the set of integers is denoted by $\Integ$.
    \item All rings are assumed to be commutative, associative and unital. 
    \item The symbol $\bk$ is always a field of characteristic zero and we use the notation $R = \bk^{[n]}$ to mean that $R$ is a polynomial ring in $n$ variables over $\bk$.  
    \item If $X$ is a smooth projective variety of dimension $n$, we let $p_g(X)$ denote the geometric genus of $X$ and we note that $p_g(X) = \dim_\bk H^0(X, \omega_X) = \dim_\bk H^n(X, \OSheaf_X)$.
    \item Unless stated otherwise, all curves and surfaces are assumed to be irreducible and reduced over an algebraically closed field $\bk$. 
    \item If $X$ is a normal variety, we let $\Cl(X)$ denote the divisor class group of $X$ and $\CaCl(X)$ denote the Cartier class group of $X$.  
\end{itemize}

\section{Definition and Basic Properties}

\begin{notation}
    Let $\bk$ be a field of characteristic zero. Let $(a_0, \dots, a_n) \in (\Nat^+)^{n+1}$ where $n \geq 0$. Define the ring   
    $$B_{\bk; a_0, \dots, a_n} = \bk[x_0, \dots, x_n] / \lb x_0^{a_0} + \dots + x_n^{a_n} \rb.$$
    Any ring of the above form is called a Pham-Brieskorn ring over $\bk$. We may drop $\bk$ from the notation if the result is valid for any field of characteristic zero.   
\end{notation}

\begin{remarks} \label{firstRemarks} \ 
    \begin{itemize}
    
    \item If $a_i = 1$ for some $i$, then $B_{a_0, \dots, a_n}$ is isomorphic to a polynomial ring. 
    \item If $n = 0$, then $B_{a_0} = \bk[x_0]/\lb x_0^{a_0} \rb$ is a reduced ring if and only if $a_0 = 1$. 
    \item It is a general fact that if $\gcd(a,b) = 1$, then for all $\lambda \in \bk^*$, $x^a + \lambda y^b$ is an irreducible element of the polynomial ring $\bk[x,y]$. Consequently, if $n = 1$ and $\bk = \bar{\bk}$, then $B_{a_0, a_1}$ is an integral domain if and only if $\gcd(a_0, a_1) = 1$. Note that if $\bk$ is not algebraically closed, we can have $\gcd(a_0, a_1) \neq 1$ but $B_{\bk; a_0, a_1}$ is still a domain. For example $B_{\Rat; 2,2}$ is a domain, but $B_{\Rat; 3,3}$ is not.
    \item The Krull dimension of $B_{a_0, \dots, a_n}$ is $n$. 
    \end{itemize}
\end{remarks}

While there are occasionally non-trivial things to say about Pham-Brieskorn rings when $n \leq 1$, we will often leave those cases to the reader and henceforth assume tacitly that $n \geq 2$.

\subsection{Integral Domains}

The first natural question we want to answer is the following one.

\begin{question}
When is $B_{a_0, \dots, a_n}$ an integral domain?
\end{question}

Pham-Brieskorn rings are always integral domains when $n \geq 2$. To prove this, we require the following:

\begin{theorem}[\cite{lang2012algebra}, Theorem VI.9.1]\label{langThm}
    Let $m \geq 2$ and let $f \in \bk^*$. Assume that
    \begin{enumerate}[\rm(i)]
    \item for all primes $p$ such that $p \mid m$, $f \notin \bk^p$, 
    \item if $4 \mid m$ then $f \notin -4 \bk^4$.
    \end{enumerate}
    Then $x^m - f$ is irreducible in $\bk[x]$. 
\end{theorem}

\begin{proposition}\label{integralDomain}
    For $n \geq 2$, $B_{a_0, \dots, a_n}$ is an integral domain. 
\end{proposition}
\begin{proof}
    This is clear if some $a_i = 1$, so assume $a_i \geq 2$ for all $i$. Without loss of generality, we may assume $\bk$ is algebraically closed since irreducibility over $\bar{\bk}$ implies irreducibility over $\bk$. It suffices to prove that $x_0^{a_0} + \dots + x_n^{a_n}$ is irreducible in $\bk[x_0, \dots, x_n]$. Suppose that $x_0^{a_0} + \dots + x_n^{a_n}$ is reducible. Then by Gauss' Lemma $x_0^{a_0} + \dots + x_n^{a_n}$ is reducible in  $K[x_0]$, where $K = \bk(x_1, \dots, x_n)$. Write $x_0^{a_0} + \dots + x_n^{a_n} = x_0^{a_0} - f$ where $f = -(x_1^{a_1} + \dots + x_n^{a_n}) \in K^*$. By Theorem \ref{langThm} (over $K$) with $m = a_0 \geq 2$, the reducibility of $x_0^{a_0} - f$ implies that either (i) or (ii) of that theorem fails. If (ii) fails, then $4 \mid a_0$ and $f = -4q^4$ for some $q \in K$. Since $\sqrt{-1} \in \bk$ we may write $f = (2\sqrt{-1}\,q^2)^2 \in K^2$, and $4 \mid a_0$ gives $2 \mid a_0 = m$; so (i) fails as well, with the prime $2$. In either case, there exist $q' \in \bk(x_1, \dots, x_n)$ and a prime $p$  dividing $a_0$ such that $(q')^p = -x_1^{a_1} - \dots - x_n^{a_n}$. Letting $\mu$ be a $p^{th}$ root of $-1$ and letting $q = \mu q'$, we obtain $q^p = x_1^{a_1} + \dots + x_n^{a_n}$. Writing $q = \frac{u}{v}$, where $u, v \in \bk[x_1, \dots, x_n]  \setminus\{0\}$, we obtain $u^p = v^p(x_1^{a_1} + \dots + x_n^{a_n})$. It follows that $x_1^{a_1} + \dots + x_n^{a_n}$ is a $p^{th}$ power in $\bk[x_1, \dots, x_n]$ i.e. there exists some $g = g(x_1, \dots, x_n) \in \bk[x_1, \dots, x_n]$ such that $g^p = x_1^{a_1} + \dots + x_n^{a_n}$. Setting $x_3 = x_4 = \dots = x_n = 0$, it follows that $g(x_1, x_2, 0,\dots, 0)^p = x_1^{a_1} + x_2^{a_2}$. Thus, to prove the proposition, it suffices to show:
    \begin{equation}
        x_1^{a_1} + x_2^{a_2} \text{ is not an $m^{th}$ power in $\bk[x_1, x_2]$ for any $m > 1$}.
    \end{equation}
    This follows from Remark \ref{firstRemarks} if $\gcd(a_1, a_2) = 1$, so we may assume that $\gcd(a_1,a_2) = d > 1$. Let $\alpha_1 = a_1 / d$ and let $\alpha_2 = a_2 / d$. Then $$x_1^{a_1} + x_2^{a_2} = (x_1^{\alpha_1})^d + (x_2^{\alpha_2})^d = \prod_{i=1}^d(x_1^{\alpha_1} + \mu_i {x_2}^{\alpha_2})$$ where the $\mu_i$ are distinct elements of $\bk^*$. Furthermore, this is a factorization of $x_1^{a_1} + x_2^{a_2}$ into non-associate irreducible components by Remark \ref{firstRemarks}. Since $\bk[x_1,x_2]$ is a UFD, $x_1^{a_1} + x_2^{a_2}$ is not an $m^{th}$ power in $\bk[x_1, x_2]$, as required.  
\end{proof}

\subsection{Normality}

In this subsection, we will show that for $n \geq 2$, $B_{a_0, \dots, a_n}$ is normal.

\begin{remark}\label{uniqueSingularPoint}
    Using the Jacobian criterion (assuming $a_i \geq 2$ for all $i$), it is easy to verify that $\mgoth = \lb x_0, \dots, x_n \rb \in \Spec B_{a_0, \dots, a_n}$ is the unique singular point of $\Spec B_{a_0, \dots, a_n}$.  
\end{remark}

\begin{definitions}
Recall that a \textit{normal domain} is an integral domain that is integrally closed in its field of fractions and a ring $R$ is \textit{normal} if $R_\pgoth$ is a normal domain for every prime ideal $\pgoth \in \Spec R$. A Noetherian local ring $R$ is \textit{Cohen-Macaulay} if the depth of $R$ (as a module over itself) is
equal to the Krull dimension of $R$. A ring $R$ is \textit{Cohen-Macaulay} if it is Noetherian and if $R_\pgoth$ is
Cohen-Macaulay for all $\pgoth \in \Spec R$. If $R$ is a ring and $k \in \Nat$, then $R$ \textit{satisfies condition $S_k$} if $\depth R_\pgoth \geq \inf\{k, \height \pgoth\}$
for every prime ideal $\pgoth \in \Spec R$. A ring $R$ is \textit{regular in codimension $k$} if $R_\pgoth$ is a regular local ring for all prime
ideals $\pgoth$ of $R$ satisfying $\height \pgoth \leq k$.
\end{definitions}
The following results are useful. 
\begin{remarks}\label{CMRemarks} \ 
\begin{enumerate}[\rm(i)]
\item Every regular ring is Cohen-Macaulay. 
\item A Noetherian ring $R$ is Cohen-Macaulay if and only if $R$ satisfies $S_k$ for all $k \in \Nat$.\cite[Example B.77]{gortz2020algebraic}
\item If $R$ is Cohen-Macaulay, so is the polynomial ring $R[x]$. 
\end{enumerate}
\end{remarks}

\begin{lemma}[\cite{MR1251956}, Theorem 2.1.3 (a)]\label{quotientRegularCM}
Let $S$ be a Noetherian Cohen-Macaulay ring and let $I = \lb f_1, \dots, f_k \rb$ be such
that $(f_1, \dots , f_k)$ is an $S$-regular sequence. Then $S/I$ is Cohen-Macaulay.
\end{lemma}

\begin{proposition}\label{CM}
    If $R$ is a Noetherian Cohen-Macaulay integral domain, then for every $f \in R[x_1, \dots, x_n]$, the ring $R[x_1, \dots, x_n] / \lb f \rb$ is Cohen-Macaulay. 
\end{proposition}

\begin{proof} Since $R$ is Cohen-Macaulay, so is $R[x_1, \dots, x_n]$. As such, the result is true both when $f = 0$ and when $f$ is a unit. Otherwise, $(f)$ is a regular sequence, so the result follows from Lemma \ref{quotientRegularCM}. 
\end{proof}

\begin{theorem}[Serre's Normality Criterion]\label{Serre}
    A Noetherian ring $R$ is normal if and only if the following hold:
    \begin{enumerate}[\rm(i)]
        \item $R$ is regular in codimension 1
        \item $R$ satisfies condition $S_2$.
    \end{enumerate}
\end{theorem}

\begin{corollary}\label{PBNormal}
    Let $n \geq 2$. Then $B_{a_0, \dots, a_n}$ is normal. 
\end{corollary}
\begin{proof}
    This is clear if some $a_i = 1$ so assume $a_i \geq 2$ for all $i$. Apply Theorem \ref{Serre}. Part (i) of Theorem \ref{Serre} is clearly satisfied because $\Sing \Spec B_{a_0, \dots, a_n} = \{\mgoth\}$ with $\height \mgoth = n \geq 2$  and part (ii) follows from Proposition \ref{CM} together with Remark \ref{CMRemarks} (ii). 
\end{proof}

\section{$\Nat$-Grading Structure}
Recall that a ring $R$ is $\Nat$-graded if we can write $R = \bigoplus_{d \in \Nat} R_d$ where each $R_d$ is an abelian group such that $R_m R_n \subseteq R_{m+n}$. For any $d \in \Nat$, the elements of $R_d$ are said to be \textit{homogeneous of degree $d$}. Given $r \in R$, the unique decomposition $r = \sum_{d \in \Nat } r_d$ with finitely many $r_d \neq 0$ is called the \textit{homogeneous decomposition of $r$}. For nonzero $r$, the largest $d$ such that $r_d \neq 0$ is called the degree of $r$; the degree of 0 is $-\infty$.  

    If $R = \bigoplus_{d \in \Nat} R_d$ is an $\Nat$-graded ring, we say that an ideal $I \lhd R$ is \textit{homogeneous} if it can be generated by homogeneous elements. If $I$ is a homogeneous ideal, then the quotient ring $R / I$ admits a well-defined grading $R/I = \bigoplus_{d \in \Nat} (R_d + I) / I$. 

\begin{example}\label{PBGradingExample}

    Given an $(n+1)$-tuple $(a_0, \dots, a_n) \in (\Nat^+)^{n+1}$, let $L = \lcm(a_0, \dots, a_n)$ and for each $i = 0,1, \dots, n$, define $w_i = L / a_i$. Consider the graded polynomial ring $R = \bk[x_0, \dots, x_n]$ whose grading $\ggoth$ is defined by $\bk \subseteq R_0$, $\deg(x_i) = w_i$ for each $i = 0, \dots, n$ and observe that the element $f = x_0^{a_0} + \dots + x_n^{a_n}$ is $\ggoth$-homogeneous of degree $L$. Consequently, the ideal $\lb f \rb \lhd \bk[x_0, \dots, x_n]$ is a $\ggoth$-homogeneous ideal and so $B_{a_0, \dots, a_n} = \bk[x_0, \dots, x_n] / \lb f \rb$ is an $\Nat$-graded ring satisfying $\deg(\bar{x}_i) = w_i$.
\end{example}

\begin{definition}\label{goodCstar}
    An algebraic variety $V$ admits a \textit{good $\bk^*$-action} if there is a closed embedding $V \hookrightarrow \Aff^{n+1}$ and a $\bk^*$-action $t \cdot (z_0, z_1, \dots, z_n) = (t^{w_0} z_0, \dots, t^{w_n} z_n)$ with each $w_i > 0$ that leaves $V$ invariant under this action. 
\end{definition}
\begin{remark}\label{goodCStarRemark}
Algebraically, if $V = \Spec A$, the definition of $V$ admitting a good $\bk^*$-action is equivalent to requiring that $A$ admits a positive $\Nat$-grading, i.e., $A = \bigoplus_{i \in \Nat} A_i$ with $A_0 = \bk$. Example \ref{PBGradingExample} shows that if $V = \Spec B_{a_0, \dots, a_n}$, then $V$ does indeed admit a good $\bk^*$-action.
\end{remark}

\section{Isomorphism Classes and Cancellation}

It is not too hard to show that over a field of characteristic zero Pham-Brieskorn rings that are not polynomial rings are determined up to isomorphism by their exponents. More precisely, we have 
\begin{proposition}\cite[Theorem 7]{hajra2026generalizedzariskicancellationbrieskornpham}\label{PBExponents}
    Let $\bk$ be a field of characteristic zero, let $(a_0, \dots, a_n), (b_0, \dots, b_n) \in (\Nat_{\geq 2})^{n+1}$. Then the following are equivalent:
    \begin{enumerate}[\rm(i)]
        \item $B_{\bk; a_0, \dots, a_n} \isom B_{\bk; b_0, \dots, b_n}$
        \item up to a permutation of the entries, $(a_0, \dots, a_n) = (b_0, \dots, b_n)$.
    \end{enumerate}
\end{proposition}
Let us give some comments about how (i)$\Rightarrow$(ii). (The (ii)$\Rightarrow$(i) part of the proof is obvious.) The assumption that $B_{\bk; a_0, \dots, a_n} \isom B_{\bk; b_0, \dots, b_n}$ implies that any isomorphism $\varphi : B_{\bk; a_0, \dots, a_n} \to B_{\bk; b_0, \dots, b_n}$ induces an isomorphism of schemes between their respective singular points. This implies that $\bk[x_0, \dots, x_n] / \lb x_0^{a_0 - 1}, \dots, x_n^{a_n-1} \rb \isom \bk[x_0, \dots, x_n] / \lb x_0^{b_0 - 1}, \dots, x_n^{b_n-1} \rb$ (where we use that $x_0^{a_0} + \dots + x_n^{a_n} \in \lb x_0^{a_0 - 1}, \dots, x_n^{a_n-1} \rb$) from which one deduces (see \cite[Lemma 6]{hajra2026generalizedzariskicancellationbrieskornpham}) that (up to a permutation) $a_i = b_i$ for all $i$.

The article \cite{hajra2026generalizedzariskicancellationbrieskornpham} in fact proves that certain singular varieties (including the family of Pham-Brieskorn varieties) satisfy a nice cancellation property. More generally, the authors prove:

\begin{theorem}\cite[Theorem 13]{hajra2026generalizedzariskicancellationbrieskornpham}\label{PBCancellation}
    Let $R$ and $S$ be positively graded affine domains over $\Comp$. Let $V = \Spec R$ and $W = \Spec S$ be the corresponding varieties admitting good $\Comp^*$-actions with vertices $p$ and $q$ respectively. Assume $p$ and $q$ are the unique singularities of $V$ and $W$, respectively. Let $Z$ be a separated scheme over $\Comp$ (not necessarily connected) having a smooth point such that $V \times  Z \isom_\Comp W \times Z$. Then $V$ and $W$ are isomorphic as affine $\Comp$ varieties. 
\end{theorem}

We deduce the following minor generalization of \cite[Theorem 14]{hajra2026generalizedzariskicancellationbrieskornpham}. We will give the entire proof of the result below in order to demonstrate a useful method of  generalizing results over $\Comp$ to an arbitrary field of characteristic zero. It is often the case that such arguments are omitted and authors may write something like ``a Lefschetz principle argument shows that the claim generalizes to any field of characteristic zero". In such cases, presumably the argument will follow a similar track to the one given in Corollary \ref{generalCancellation}. We first recall the following basic lemma. 

\begin{lemma}\label{fieldEmbedding}
    Every finitely generated field extension of $\Rat$ can be embedded in $\Comp$. 
\end{lemma}

\begin{corollary}\label{generalCancellation}
    Let $n \geq 2$ and let $\bk$ be a field of characteristic zero. If
    $(B_{\bk; a_0, \dots, a_n})^{[c]} \isom_\bk (B_{\bk; b_0, \dots, b_m})^{[d]}$ for some
    $c, d \in \Nat$ and no $a_i = 1$, then $c = d$, $m = n$ and up to a permutation of the
    entries $(a_0, \dots, a_n) = (b_0, \dots, b_n)$.
\end{corollary}
\begin{proof}
    We first prove the case $\bk = \Comp$. Considering dimensions, $n + c = m + d$. It is easy to see that 
    $\Sing \left(\Spec  (B_{\Comp; a_0, \dots, a_n})^{[c]} \right)$
    is nonempty and $c$-dimensional. The assumed isomorphism easily implies that no $b_{j} = 1$, that $m \geq 1$, 
    that $\Sing \left( \Spec (B_{\Comp; b_0, \dots, b_m})^{[d]} \right)$ is $d$-dimensional and hence that $c = d$ and, by the equality of dimensions, that $m = n$.

    Since $B_{\Comp; \ba}$ and $B_{\Comp; \bb}$ are positively graded by Example \ref{PBGradingExample} with unique singular points at their respective origins, we can apply Theorem \ref{PBCancellation} with $Z = \Aff_\Comp^c$ to
    $R = B_{\Comp; a_0, \dots, a_n}$ and $S = B_{\Comp; b_0, \dots, b_n}$. It follows that 
    $B_{\Comp; a_0, \dots, a_n} \isom B_{\Comp; b_0, \dots, b_n}$ and by
    Proposition \ref{PBExponents} we obtain $(a_0, \dots, a_n) = (b_0, \dots, b_n)$ up to a permutation of the entries.

    To prove the Corollary, it now suffices to prove the following claim:
    \begin{equation}\label{reductionAnyk}
        \text{if $(B_{\bk; a_0, \dots, a_n})^{[c]} \isom_\bk (B_{\bk; b_0, \dots, b_m})^{[d]}$
        then $(B_{\Comp; a_0, \dots, a_n})^{[c]} \isom_{\Comp} (B_{\Comp; b_0, \dots, b_m})^{[d]}$}
    \end{equation}

    Let $P= \bk[x_0, \dots, x_n, u_1, \dots, u_c] = \bk^{[n+c+1]}$ and
    $Q = \bk[y_0, \dots, y_m, v_1, \dots, v_d] = \bk^{[m+d+1]}$ and let us abbreviate
    $f_\ba = x_0^{a_0} + \dots + x_n^{a_n} \in \bk[x_0, \dots, x_n]$ and
    $f_\bb = y_0^{b_0} + \dots + y_m^{b_m}$. Let $\varphi$ be an isomorphism $P/\lb f_\ba \rb \to Q/\lb f_\bb \rb$ whose inverse is $\psi$. For each
    $i = 0 , \dots, n$ and each $j = 1, \dots, c$, let $s_i \in Q$ be a representative of $\varphi(\bar{x}_i) \in Q/\lb f_\bb \rb$ and let
    $t_j \in Q$ be a representative of $\varphi(\bar{u}_j)$. Conversely, for each
    $k = 0, \dots, m$ and each $\ell = 1, \dots, d$ let $w_k \in  P$ be a
    representative of $\psi(\bar{y}_k)$ and let $z_\ell \in P$ be a representative of
    $\psi(\bar{v}_\ell)$.

    Let $K \subseteq \bk$ denote the subfield of $\bk$ generated over $\Rat$ by the
    coefficients of all of the $s_i, t_j \in Q$ and the $w_k, z_\ell \in P$.
    Since there are finitely many $s_i, t_j, w_k, z_\ell$, each with finitely many
    coefficients, $K$ is a finitely generated field extension of $\Rat$. Since $f_\ba$ has coefficients in $\Rat \subseteq K$,
    we have $B_{K; \ba}[u_1, \dots, u_c] \otimes_K \bk \isom B_{\bk; \ba}[u_1, \dots, u_c]$ and as $\bk$ is free as a $K$-module on a basis containing $1$, the canonical map into this extension is injective, so we may view $B_{K; \ba}[u_1, \dots, u_c]$ as a subring of $B_{\bk; a_0, \dots, a_n}[u_1, \dots, u_c]$. By our choice of $K$, it can be checked that $\varphi$ and $\psi$ restrict to
    $K$-isomorphisms
    $\varphi_0 : B_{K; a_0, \dots, a_n}[u_1, \dots, u_c] \to B_{K; b_0, \dots, b_m}[v_1, \dots, v_d]$
    and
    $\psi_0 : B_{K; b_0, \dots, b_m}[v_1, \dots, v_d] \to B_{K; a_0, \dots, a_n}[u_1, \dots, u_c]$
    that are mutually inverse.

    Since $K$ is finitely generated over $\Rat$, Lemma \ref{fieldEmbedding} implies that we may assume
    without loss of generality that $\Rat \subseteq K \subseteq \Comp$. Extension of scalars $- \otimes _K \Comp$ commutes with adjoining variables,
    and $B_{K; a_0, \dots, a_n} \otimes_K \Comp \isom B_{\Comp; a_0, \dots, a_n}$ since $f_\ba$ has coefficients in
    $\Rat$. Applying $- \otimes_K \Comp$ to $\varphi_0$ therefore
    gives $B_{\Comp; a_0, \dots, a_n}[u_1, \dots, u_c] \isom
    B_{\Comp; b_0, \dots, b_m}[v_1, \dots, v_d]$, proving \eqref{reductionAnyk} and completing
    the proof.
\end{proof}

\subsection{Isomorphism classes of $V(x_0^{a_0} + \dots + x_n^{a_n} + 1)$}

Given Proposition \ref{PBExponents}, one might ask the following question:

\begin{question}\label{variantQuestion}
    Assume $(a_0, \dots, a_n), (b_0, \dots, b_n)  \in (\Nat_{\geq 2})^{n+1}$ and suppose 
    $$\bk[x_0, \dots, x_n] / \lb x_0^{a_0} + \dots + x_n^{a_n}+1 \rb \isom \bk[y_0, \dots, y_n] / \lb y_0^{b_0} + \dots + y_n^{b_n}+1 \rb.$$ Is $(a_0, \dots, a_n) = (b_0, \dots b_n)$ up to a permutation of the entries?
\end{question}  

The main difference between this family of varieties and the family of Pham-Brieskorn varieties is that these varieties are smooth and so we don't have the invariants associated with singular points to help us detect differences. In \cite{gurjar2025classification} the authors show that Question \ref{variantQuestion} has an affirmative answer for $n = 1$ (the case of affine curves). In \cite{chitayat2026isomorphism}, we show the same for $n = 2$.

\section{Unique Factorization and Divisor Class Groups}

    Now that we know that Pham-Brieskorn rings are normal and Cohen-Macaulay, the following question arises naturally.

\begin{question}\label{UFDQuestion}
    Which Pham-Brieskorn rings are unique factorization domains?
\end{question}

Question \ref{UFDQuestion} is answered by Proposition \ref{PBUFDDim2} for $n = 2$ and $\bk = \bar{\bk}$, by Theorem \ref{storchUFD} for $n = 3$ and $\bk = \Comp$ and by Proposition \ref{highDimUFD} for $n \geq 4$ and arbitrary $\bk$ of characteristic zero. It is also worth noting that since $B_{a_0, \dots, a_n}$ is $\Nat$-graded over $\bk$, \cite[Remark 8.2]{samuel1964ufd} implies  
\begin{remark}\label{baseChangeUFD}
    If $\bar{\bk}$ is the algebraic closure of $\bk$ and $B_{\bar{\bk}; a_0, \dots, a_n}$ is a UFD, then so is $B_{\bk; a_0, \dots, a_n}$.  
\end{remark}

\begin{proposition}\label{PBUFDDim2}
    Assume $a_0, a_1, a_2 \geq 2$ and that $\bk = \bar{\bk}$. The Pham-Brieskorn ring $B_{\bk; a_0, a_1, a_2}$ is a UFD if and only if $a_0, a_1, a_2$ are pairwise relatively prime.  
\end{proposition}
\begin{proof}
    Let us write $B = B_{\bk; a_0, a_1, a_2} = \bk[x_0, x_1, x_2]  / \lb x_0^{a_0} + x_1^{a_1} + x_2^{a_2} \rb$ and let $\bar{x}_i$ denote the canonical image of $x_i$ in $B$. The $(\Leftarrow)$ direction is given by \cite[Example 1, p.26]{samuel1964ufd}. We prove $(\Rightarrow)$ by contrapositive. By assumption, we may assume without loss of generality that $\gcd(a_0,a_1) \neq 1$. It can be checked that $\bar{x}_2$ is irreducible in $B$ (use that $\bar{x}_2$ is $\ggoth$-homogeneous with respect to the grading defined in Example \ref{PBGradingExample}). Since $B_{\bk; a_0, a_1, a_2} / \lb \bar{x}_2 \rb \isom B_{\bk; a_0, a_1}$ is not a domain (by Remark \ref{firstRemarks}), it follows that $\bar{x}_2$ is not prime in $B$ and so $B_{a_0, a_1, a_2}$ is not a UFD.
\end{proof}

For the rest of this section we will use the theory of divisor class groups of varieties over an algebraically closed $\bk$. Since by Corollary \ref{PBNormal}, $\Spec B_{a_0, \dots, a_n}$ is normal, it makes sense to define $\Cl(\Spec B_{\bk; a_0, \dots,  a_n})$ whenever $\bk = \bar{\bk}$. We make use of the following theorem. 

\begin{theorem}[\cite{Hartshorne}, Proposition II.6.2]
    Let $A$ be a Noetherian $\bk$-domain where $\bk = \bar{\bk}$ and let $X = \Spec A$. Then $A$ is a unique factorization domain if and only if $A$ is normal and $\Cl(X) = 0$. 
\end{theorem}

Over $\Comp$, the divisor class group of $\Spec B_{\Comp; a_0, a_1, a_2}$ can be calculated using the following result.     

\begin{theorem}\cite[Theorem 1.1, Pham-Brieskorn case]{WROBEL202043}. 
    Let $X = \Spec B_{\Comp; a_0, a_1, a_2}$ where $a_0, a_1, a_2 \geq 2$. Assume without loss of generality that $\gcd(a_0, a_1) \geq \gcd(a_0, a_2)$ and $\gcd(a_0, a_1) \geq \gcd(a_1, a_2)$. The following hold: 
    \begin{enumerate}[\rm(a)]
        \item $\Cl(X)$ is trivial if and only if $\gcd(a_i,a_j) = 1$ for all $i\neq j$. 
        \item If $g_{0,1} = \gcd(a_0,a_1)>1$ and $\gcd(a_0, a_2) = \gcd(a_1,a_2) = 1$, then $\Cl(X) \isom (\Integ / a_2 \Integ )^{g_{0,1}-1}.$
        \item If $\gcd(a_0, a_1) = \gcd(a_0,a_2) = \gcd(a_1,a_2) = 2$, then $\Cl(X) \isom \Integ / (a_0a_1a_2 / 4)\Integ$.
        \item If none of (a)-(c) are satisfied, then $\Cl(X)$ is not finitely generated. 
    \end{enumerate}
\end{theorem}

In order to determine when $B_{\Comp; a_0, a_1, a_2, a_3}$ is a UFD, we require some notation. Note that the goal is to state Theorem \ref{PBCL4} and Theorem \ref{storchUFD}, two of the central results of \cite{storch1984picard}. Note that we modify the notation slightly, to preserve our notation. We will use $\ba = (a_0, a_1, a_2, a_3)$ whereas Storch uses  $\textbf{r} = (r_1, r_2, r_3, r_4)$. 
\begin{notation}
    Consider the tuple $\textbf{a} = (a_0, a_1, a_2, a_3)$. For each subset $H \subseteq \{0,1,2,3\}$ define 
    $$\textbf{a}_H = \prod_{i \in H} a_i, \qquad  \textbf{v}_H = \underset{i \in H}\lcm(a_i), \qquad \tau(\ba) = \sum_{H \subseteq \{0,1,2,3\}} (-1)^{|H|} \frac{\textbf{a}_H}{\textbf{v}_H}$$
    with conventions $\ba_\emptyset = 1$ and $\textbf{v}_\emptyset = 1$. 
    Also, given a set of fractions $\{\frac{p_0}{q_0},\frac{p_1}{q_1},  \dots, \frac{p_n}{q_n}\}$ expressed in lowest terms, we can define $\gcd(\frac{p_0}{q_0},\frac{p_1}{q_1},  \dots, \frac{p_n}{q_n}) = \frac{\gcd(p_0, \dots, p_n)}{\lcm(q_0, \dots, q_n)}$.  
\end{notation}

\begin{nothing}
    Let $S_0$ denote the set of 4-tuples $(\alpha_0, \alpha_1,\alpha_2, \alpha_3) \in (\Nat^+)^4$ such that $\frac{\alpha_0}{a_0} + \frac{\alpha_1}{a_1} + \frac{\alpha_2}{a_2} + \frac{\alpha_3}{a_3} = 1$ and declare two elements $(\alpha_0, \alpha_1,\alpha_2, \alpha_3)$, $(\beta_0, \beta_1,\beta_2, \beta_3)$ of $S_0$ to be equivalent if there exists some $s$ coprime to $t = \lcm(a_0, a_1, a_2, a_3)$ such that $s\left(\frac{\alpha_0}{a_0},  \frac{\alpha_1}{a_1} , \frac{\alpha_2}{a_2} , \frac{\alpha_3}{a_3} \right) = \left(\frac{\beta_0}{a_0}, \frac{\beta_1}{a_1}, \frac{\beta_2}{a_2}, \frac{\beta_3}{a_3} \right) \mod \Integ^4$. It can be checked that if $\frac{\alpha_0}{a_0} + \frac{\alpha_1}{a_1} + \frac{\alpha_2}{a_2} + \frac{\alpha_3}{a_3} = 1$ then $\gcd\left(\frac{\alpha_0}{a_0},  \frac{\alpha_1}{a_1} , \frac{\alpha_2}{a_2} , \frac{\alpha_3}{a_3} \right) = \frac{1}{d}$ for some $d$ dividing $t$ and that if $\left(\frac{\alpha_0}{a_0},  \frac{\alpha_1}{a_1} , \frac{\alpha_2}{a_2} , \frac{\alpha_3}{a_3} \right) \sim \left(\frac{\beta_0}{a_0},  \frac{\beta_1}{a_1} , \frac{\beta_2}{a_2} , \frac{\beta_3}{a_3} \right)$ then their $\gcd$s are equal.
\end{nothing}

\begin{example}\label{S0remarks}
    Consider the case where $a_0 = a_1 = a_2 = a_3 = 7$. Then $(1,1,1,4)$ and $(2,2,2,1)$ are equivalent elements of $S_0$ since $2 \cdot \left(\frac{1}{7},  \frac{1}{7} , \frac{1}{7} , \frac{4}{7} \right) \equiv \left(\frac{2}{7},  \frac{2}{7} , \frac{2}{7} , \frac{8}{7} \right) \equiv \left(\frac{2}{7},  \frac{2}{7} , \frac{2}{7} , \frac{1}{7} \right) \mod \Integ^4$. For general $(a_0, a_1, a_2, a_3)$, computing the set $S_0$ and its equivalence classes is not always easy to do by hand. When the $a_i$ are small however, it is often possible to do the computations by hand. We demonstrate this below.  

    Consider the case where $\textbf{a} = (2,3,5,30)$. We will see that this case is of interest because of its appearance in Section \ref{rigiditySection}. Since $\sum_{i = 0}^3 {\frac{1}{a_i}} > 1$, $S_0 = \emptyset$. In the special case where $\sum_{i = 0}^3 {\frac{1}{a_i}} = 1$, $S_0$ has $(1,1,1,1)$ as its unique element and hence contains one equivalence class. 
\end{example}

\begin{theorem}\label{PBCL4}\cite[Satz 2.1]{storch1984picard}
    Let $X = \Spec B_{\Comp; a_0,a_1, a_2, a_3}$ and let $t = \lcm(a_0, a_1, a_2, a_3)$. Then $$\Cl(X) = \Integ^{\rho(\ba)} \text{ where } \rho(\ba) = \tau(\ba) - \sum_{d \mid t} \varphi(d) w(d;\ba)$$  
    where $\varphi(d)$ is Euler's phi function and $w(d;\ba)$ is the number of equivalence classes of $S_0$ whose $\gcd$ is exactly $\frac{1}{d}$. 
\end{theorem}

\begin{example}
    Let us compute the divisor class group of $\Spec(B_{\Comp;4,4,4,4})$. For each nonempty subset $H \subseteq \{0,1,2,3\}$ we have $a_H = 4^{|H|}$, $v_H = 4$. Then 

    $$\tau(\ba) = 1 - 4(\frac{4}{4}) + 6(\frac{16}{4}) - 4(\frac{64}{4}) + 64 = 21.$$

    By Example \ref{S0remarks}, $S_0$ consists of one element because $\sum_{i = 0}^3\frac{1}{a_i} = 1$ and the $\gcd(\frac{1}{4},\frac{1}{4},\frac{1}{4},\frac{1}{4}) = \frac{1}{4}$. Thus the only non-zero summand in the sum $\sum_{d \mid t} \varphi(d) w(d;\ba)$ when $\ba = (4,4,4,4)$ is $d=4$, in which case we have $\varphi(d) = 2$ and $w(d; \ba) = 1$. We deduce that $\rho(\ba) = 21- 2 = 19$ and $\Cl(\Spec B_{\Comp; 4,4,4,4}) = \Integ^{19}$. 
    
\end{example}

\noindent \textbf{The graph of $\bf{a}$}. Given $\textbf{a} = (a_0, a_1, a_2, a_3)$ we assign a decorated graph of 4 vertices, one vertex $v_i$ associated to each $a_i$. We join $v_i$ and $v_j$ by an edge $\varepsilon_{ij}$ if and only if $\gcd(a_i, a_j) \neq 1$ and we decorate this edge $\varepsilon_{ij}$ with the positive integer $\gcd(a_i, a_j)$. 

\begin{theorem}\cite[Satz 2.3]{storch1984picard}\label{storchUFD}
Let ${\bf a}= (a_0,a_1,a_2,a_3)\in (\mathbb{N}^+)^4$. 
Then $B_{\Comp; a_0, a_1, a_2, a_3}$ is a UFD if and only if $\textbf{a}$
satisfies one of the following two conditions:
\begin{enumerate}[\rm(i)]
  \item one of the $a_i$ is coprime to the remaining three;
  \item after a suitable renumbering of the $a_i$ the graph of $\textbf{a}$ has the shape
    
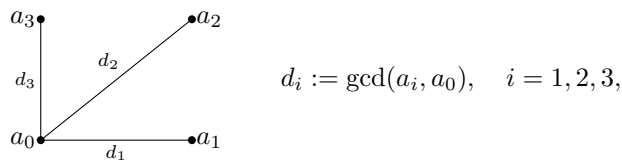
\[
      \begin{tikzpicture}[baseline=(current bounding box.center),
                          every node/.style={inner sep=1.5pt}]
        \coordinate (v4) at (0,0);
        \coordinate (v1) at (2,0);
        \coordinate (v3) at (0,1.6);
        \coordinate (v2) at (2,1.6);
        \draw (v4) -- node[below] {$\scriptstyle d_1$} (v1);
        \draw (v4) -- node[left]  {$\scriptstyle d_3$} (v3);
        \draw (v4) -- node[above left, pos=0.55] {$\scriptstyle d_2$} (v2);
        \fill (v1) circle (1.5pt) node[right] {$a_1$};
        \fill (v2) circle (1.5pt) node[right] {$a_2$};
        \fill (v3) circle (1.5pt) node[left]  {$a_3$};
        \fill (v4) circle (1.5pt) node[left]  {$a_0$};
      \end{tikzpicture}
      \qquad d_i:=\gcd(a_i,a_0),\quad i=1,2,3,
    \]
    and there is no permutation $\pi\in\mathfrak{S}_3$ with
    $2\mid d_{\pi(1)}$, $3\mid d_{\pi(2)}$, $5\mid d_{\pi(3)}$.
\end{enumerate}
If the graph of $\textbf{a}$ has the shape in (ii) and such a permutation $\pi$ does exist, then $\rho(\ba)=8$.
\end{theorem} 

There are a few big results that go into proving Proposition \ref{highDimUFD}, the main result of this section. These results are useful on their own and so we include the statements here. 

\begin{theorem}[\cite{SGA2}, Corollaire XI.3.14]\label{UFDlocal}
    Suppose $A$ is a local Noetherian ring that is a complete intersection. If $A_\pgoth$ is a unique factorization domain for all $\pgoth$ such that $\dim A_\pgoth \leq 3$, then $A$ is a unique factorization domain.   
\end{theorem}

\begin{theorem}[\cite{Hartshorne}, Corollary II.6.16]\label{ClPic}
    If $X$ is a Noetherian, integral, separated, locally factorial scheme, then $\Cl(X) \isom \Pic(X)$.  
\end{theorem}

\begin{lemma}[\cite{murthy1969vector}, Lemma 5.1]\label{murthyLemma}
Let $A$ be a normal graded algebra with $A_0 = \bk$, generated by finitely many elements of positive degree. Then $\Pic(A) =  0$.
\end{lemma}
\begin{remark}
Note that in the statement of Lemma \ref{murthyLemma}, $A$ is not necessarily a polynomial ring. Also, if $X = \Spec A$, then $\Pic(X) = \Pic(A)$.
\end{remark}

We can now prove (for any field of characteristic zero):

\begin{proposition}\label{highDimUFD}
    Let $n \geq 4$. Then $B_{\bk; a_0, \dots, a_n}$ is a unique factorization domain. 
\end{proposition}

\begin{proof}
    Let $B = B_{\bk; a_0, \dots, a_n}$. By Remark \ref{baseChangeUFD}, it suffices to prove the special case where $\bk$ is algebraically closed. We first prove that $X = \Spec B$ is locally factorial, i.e. that $B_\pgoth$ is a unique factorization domain for all $\pgoth \in \Spec B$. Let $\mgoth$ denote the maximal ideal $\lb x_0, \dots, x_n \rb \in \Spec B$. For all $\pgoth \in \Spec B \setminus \{\mgoth\}$ (which includes all primes of height at most 3), $B_\pgoth$ is regular and hence is a unique factorization domain.  Since $B_\mgoth$ is a local Noetherian complete intersection, Theorem \ref{UFDlocal} implies that $B_\mgoth$ is a UFD. It follows that $B$ is locally factorial. Consequently, $X$ satisfies the assumptions of Theorem \ref{ClPic} and so $\Cl(X) = \Pic(X) = \Pic(B) = 0$, the last equality by Lemma \ref{murthyLemma} because $B$ is positively graded by Example \ref{PBGradingExample}.  
\end{proof}

\section{Remarks on the Singular Point}

In this section we discuss the unique singular point of $\Spec B_{a_0, \dots, a_n}$ for $n = 1$ and $n = 2$. Throughout this section, we assume $\bk = \bar{\bk}$. 

\subsection{The case of curves $(\bk = \bar{\bk})$}

\begin{nothing}\label{nequals2}  When $n = 1$, we restrict to the case where $B_{\bk; a_0, a_1}$ is an integral domain, in which case (since $\bk = \bar{\bk}$) we have $\gcd(a_0, a_1) = 1$. Since normalization gives a resolution of singularities, we need only compute the normalization of $\Spec(B_{\bk; a_0,a_1})$. We will show that the normalization of $B_{\bk; a_0, a_1}$ is an affine line. Equivalently, we show that the integral closure of $B_{a_0, a_1}$ in its field of fractions is a polynomial ring in 1 variable. 

Since $\bk = \bar{\bk}$, we have an isomorphism $B_{a_0, a_1} \isom \bk[x_0, x_1]  /\lb x_0^{a_0} - x_1^{a_1} \rb$. We will let $C = \bk[x_0, x_1]  /\lb x_0^{a_0} - x_1^{a_1} \rb$ for the rest of this example and will compute the integral closure of $C$. Observe that we have an isomorphism $\phi: C \overset{\isom}{\rightarrow} \bk[t^{a_0}, t^{a_1}] \subset \bk[t]$ given by $\bar{x}_0 \mapsto t^{a_1}$ and $\bar{x}_1 \mapsto t^{a_0}$ and that $\Frac(\bk[t^{a_0}, t^{a_1}]) = \bk(t)$ because $\gcd(a_0, a_1) = 1$. Consequently, we can extend $\phi$ to an isomorphism $\bar{\phi} : \Frac{C} \to \bk(t)$. Since $\bk[t]$ is the integral closure of $\bk[t^{a_0}, t^{a_1}]$ in $\bk(t)$, it follows $B_{a_0, a_1}$ is rational and its integral closure is $\bk^{[1]} = \bk[\bar{\phi}^{-1}(t)] \subset \Frac C$. 
\end{nothing}

\subsection{The case of surfaces ($\bk = \Comp$)}

When $n = 2$, the singular point of $\Spec B_{\Comp; a_0, a_1, a_2}$ is more complicated to resolve. To do so requires the more general theory of singularities of algebraic surfaces with ``good $\Comp^*$ actions" developed by Orlik and Wagreich in \cite{orlik1979} and \cite{Orlik1977AlgebraicSW}. Recall Definition \ref{goodCstar} and Remark \ref{goodCStarRemark}. 

\begin{nothing}\label{orlikSetup}
	Assume an algebraic surface $V$ admits a good $\Comp^*$-action and has a unique singular point $o$. Then, as discussed on pages 710--711 of \cite{orlik1979}, one can construct a 3-manifold called ``the link of the singularity $(V, o)$'' that contains enough information to fully describe the minimal resolution graph of the exceptional curve $\pi^{-1}(o)$ of a minimal resolution $\pi: \tilde{V} \to V$. The invariants of the link consist of a ``Betti number" $b \in \Integ$, a ``genus" $g \in \Nat$ and ``orbit invariants" $(\alpha_i, \beta_i)$ where $\alpha_i, \beta_i \in \Nat$ and $0 <\beta_i < \alpha_i$. This collection of invariants is encompassed by the following notation:
	$$ K(V, o) = \left\{ b ; g; (\alpha_1, \beta_1), \dots , (\alpha_r, \beta_r) \right\}. $$ 
    By Remark \ref{goodCStarRemark} together with Example \ref{PBGradingExample}, $V = \Spec B_{a_0, a_1, a_2}$ admits a good $\Comp^*$ action and furthermore, there are simple formulas \cite[p.711]{orlik1979}  for calculating $K(V, o)$ in terms of $a_0,a_1,a_2$. The Theorem on p.710 of \cite{orlik1979} then tells us how to use these values to determine the resolution graph of $\pi^{-1}(o)$.  
\end{nothing}

\begin{theorem}\label{OWMainTheorem}
    Let $V$ be an algebraic surface with a good $\Comp^*$-action and an isolated singularity at the origin. Suppose its link is
    \[
         K(V, o) = \left\{ b ; g; (\alpha_1, \beta_1), \dots , (\alpha_r, \beta_r) \right\}
    \]
    Then the weighted dual graph $\Gamma$ of an equivariant resolution of $V$ is
    the star-shaped graph depicted in Figure~\ref{fig:OWgraph}, in which the
    central curve has genus $g$ and every other curve has genus $0$. The
    integers $b_{i,j} \geq 2$ are those occurring in the continued fraction
    expansion
    \[
        \frac{\alpha_i}{\alpha_i - \beta_i}
            = b_{i,1} - \cfrac{1}{b_{i,2} - \cfrac{1}{\ddots \;-\; \cfrac{1}{b_{i,s_i}}}} ,
        \qquad 1 \leq i \leq r ,
    \]
    the length $s_i$ of the $i$-th branch being the length of this expansion.
\end{theorem}

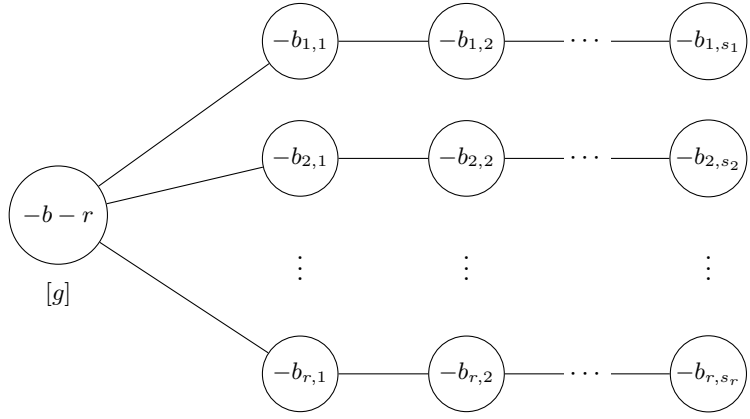
\begin{figure}[ht]
\centering
\begin{tikzpicture}[
    vertex/.style = {circle, draw, minimum size=10mm, inner sep=1pt,
                     font=\small},
    ctr/.style    = {circle, draw, minimum size=13mm, inner sep=1pt,
                     font=\small},
    x=1cm, y=1cm
]
    \node[ctr] (c) at (0,0) {$-b-r$};
    \node[font=\small] at (0,-1.05) {$[g]$};

    \node[vertex] (a11) at (3.2, 2.3) {$-b_{1,1}$};
    \node[vertex] (a12) at (5.4, 2.3) {$-b_{1,2}$};
    \node[vertex] (a1s) at (8.6, 2.3) {$-b_{1,s_1}$};
    \draw (c) -- (a11) -- (a12);
    \draw (a12) -- node[fill=white, inner sep=2pt] {$\cdots$} (a1s);

    \node[vertex] (a21) at (3.2, 0.75) {$-b_{2,1}$};
    \node[vertex] (a22) at (5.4, 0.75) {$-b_{2,2}$};
    \node[vertex] (a2s) at (8.6, 0.75) {$-b_{2,s_2}$};
    \draw (c) -- (a21) -- (a22);
    \draw (a22) -- node[fill=white, inner sep=2pt] {$\cdots$} (a2s);

    \node at (3.2, -0.6) {$\vdots$};
    \node at (5.4, -0.6) {$\vdots$};
    \node at (8.6, -0.6) {$\vdots$};

    \node[vertex] (ar1) at (3.2, -2.1) {$-b_{r,1}$};
    \node[vertex] (ar2) at (5.4, -2.1) {$-b_{r,2}$};
    \node[vertex] (ars) at (8.6, -2.1) {$-b_{r,s_r}$};
    \draw (c) -- (ar1) -- (ar2);
    \draw (ar2) -- node[fill=white, inner sep=2pt] {$\cdots$} (ars);
\end{tikzpicture}
\caption{The star-shaped resolution graph $\Gamma$ of
         Theorem~\ref{OWMainTheorem}.}
\label{fig:OWgraph}
\end{figure}

\begin{remark}\label{dontAdd10orbits}
    In the computations below, we see that orbits with a $(1,0)$ invariant can appear, and we note that the condition $0 < \beta_i < \alpha_i$ mentioned in \ref{orlikSetup} requires that we avoid adding these $(1,0)$ orbits to the link. 
\end{remark}

\begin{example}
    Let $V = \Spec B_{a_0, a_1, a_2}$. Set $d = \lcm(a_0, a_1, a_2)$, $q_i = d/a_i$ for each $i \in \{0,1,2\}$. (These $q_i$ are nothing other than the weights $w_i$ from Example \ref{PBGradingExample} although we use $q_i$ rather than $w_i$ to match the notation in \cite{orlik1979}.) Note also that $\gcd(q_0, q_1, q_2) = 1$, so if we let $\alpha_0 = \gcd(q_1, q_2)$ 
    \begin{equation}\label{uniqueBeta0}
        \text{there exists a unique $\beta_0$  such that $q_0 \beta_0 \equiv 1 \mod \alpha_0$ and $0 \leq \beta_0 < \alpha_0$ }. 
    \end{equation}
    We will use \eqref{uniqueBeta0} to determine the orbits $(\alpha_i, \beta_i)$ of the link.

    \begin{enumerate}[\rm(i)]
    \item The value of $g$ is determined by the formula 
    $$2g = \frac{d^2}{q_0q_1q_2} - \frac{d\gcd(q_0,q_1)}{q_0q_1} - \frac{d\gcd(q_1,q_2)}{q_1q_2} - \frac{d\gcd(q_0,q_2)}{q_0q_2} + \frac{\gcd(d,q_0)}{q_0} + \frac{\gcd(d,q_1)}{q_1} + \frac{\gcd(d,q_2)}{q_2} -1.$$

    \item There are $d / \lcm(q_1,q_2)$ orbits with $\alpha_0 = \gcd(q_1, q_2)$ and $\beta_0$ satisfying $q_0 \beta_0 \equiv 1 \mod \alpha_0$ and $0 \leq \beta_0 < \alpha_0$. 

    \item There are $d / \lcm(q_0,q_2)$ orbits with $\alpha_1 = \gcd(q_0, q_2)$ and $\beta_1$ satisfying $q_1 \beta_1 \equiv 1 \mod \alpha_1$ and $0 \leq \beta_1 < \alpha_1$. 

    \item There are $d / \lcm(q_0,q_1)$ orbits with $\alpha_2 = \gcd(q_0, q_1)$ and $\beta_2$ satisfying $q_2 \beta_2 \equiv 1 \mod \alpha_2$ and $0 \leq \beta_2 < \alpha_2$. 

    \item Finally, $b = \frac{d}{q_0 q_1 q_2} - \sum_{i = 1}^r \frac{\beta_i}{\alpha_i}$.
    \end{enumerate}

    As a concrete example, let us now take $V = \Spec B_{5,7,7}$. Then $d = \lcm(5,7,7) = 35$ and $(q_0, q_1, q_2) = (7,5,5)$. Part (i) gives $2g = \frac{35^2}{7 \cdot 5 \cdot 5} - \frac{35 \cdot 1}{35} - \frac{35 \cdot 5}{25} - \frac{35 \cdot 1}{35} + \frac{7}{7} + \frac{5}{5} + \frac{5}{5} - 1 = 7-1-7-1 + 1 + 1 + 1 - 1 = 0$ so $g = 0$. Part (ii) gives $7$ $(5,3)$-orbits and parts (iii) and (iv) give 2 $(1,0)$-orbits which (by Remark \ref{dontAdd10orbits}) we do not add to the link. Finally $b = \frac{35}{7 \cdot 5 \cdot 5} - 7 \cdot \frac{3}{5} = -4$. We conclude that the link is
    $$K(\Spec B_{5,7,7},o) = \{-4;0;(5,3)^7\}.$$

    Applying Theorem \ref{OWMainTheorem} gives the resolution graph:

    \begin{center}
\begin{tikzpicture}[
    vertex/.style = {circle, draw, minimum size=8mm, inner sep=1pt,
                     font=\footnotesize},
    x=1cm, y=1cm
]
    \node[vertex] (c) at (0,0) {$-3$};
    \foreach \i in {1,...,7} {
        \pgfmathsetmacro{\y}{(4-\i)*1.15}
        \node[vertex] (a\i) at (2.6, \y) {$-3$};
        \node[vertex] (b\i) at (4.6, \y) {$-2$};
        \draw (c) -- (a\i) -- (b\i);
    }
\end{tikzpicture}
\end{center}
    
\end{example}

\section{Weighted projective hypersurfaces and $\Proj(B_{a_0, \dots, a_n})$}
\subsection{Quasismooth Weighted Complete Intersections}

We collect some known results on quasismooth weighted complete intersections in weighted projective spaces. Although not strictly required, in order to be consistent with our definition of ``variety" at the start of the article, we assume that the field $\bk$ below is algebraically closed. Note that weighted projective varieties appear in many areas of algebraic geometry. The articles \cite{dolgachev} and \cite{iano-fletcher_2000} are two of the standard references. There is also a new book on the topic \cite{przyjalkowski2026weighted}.

\begin{nothing}
    If $R = \bk[x_0, x_1, \dots, x_n]$ is a polynomial ring, then it admits many $\Nat$-gradings. For each $i  \in \{0, \dots, n\}$, choose some $w_i \in \Nat^+$ and declare the element $x_i$ to be homogeneous of degree $w_i$. It is easy to see that this defines a unique grading of $R$ and that $R_d$ is an $R_0$-module with a basis consisting of monomials $M_d = \setspec{x_0^{e_0}x_1^{e_1} \cdots x_n^{e_n}}{ \sum_{i = 0}^n e_i w_i = d}$. We use the notation $\bk_{w_0, \dots, w_n}[x_0, \dots, x_n]$ to denote the polynomial ring $\bk[x_0, \dots, x_n]$ equipped with the grading satisfying $\deg(x_i) = w_i$ for all $i = 0,1, \dots, n$. 
    
    Define the \textit{weighted projective space}  $\PPP = \PPP(w_0, \dots , w_n) = \Proj(R)$. A \textit{weighted projective variety} $X$ is a closed subvariety of a weighted projective space. Whenever we write ``the variety $X \subseteq \PPP$", we mean that $X$ is a closed subvariety of $\PPP$. Since $\PPP(w_0, \dots , w_n)$ is a projective variety \cite[Corollary 4B.8]{beltramettiRobbiano}, every weighted projective variety is also a projective variety. Let $I$ be a homogeneous prime ideal of the graded ring $R$. Then $X_I = V_+(I)$ is closed in $\PPP$ and $X_I$ is isomorphic to $\Proj( R/I )$. In the special case where $I = \lb f \rb$ where $f \in R_d$, we say that $X_I = X_f$ is a \textit{weighted hypersurface} of \textit{degree $d$}. Given a weighted hypersurface $X_f \subseteq \PPP(w_0, \dots, w_n)$ of degree $d$, the \textit{amplitude} of $X_f$ is the integer $\alpha = d- \sum_{i = 0}^n w_i$.  
\end{nothing}

\begin{example}\label{pbGrading}
     Recall Example \ref{PBGradingExample}. Given $(a_0, \dots, a_n)$ and $L = \lcm(a_0, \dots, a_n)$ and $w_i = L / a_i$, the element $f = x_0^{a_0} + x_1^{a_1} + \dots + x_n^{a_n} \in \bk_{w_0, \dots, w_n}[x_0, \dots, x_n]$ is homogeneous of degree $L$, so $X_f = V_+(\lb f \rb)$ is a weighted hypersurface of $\PPP(w_0, \dots, w_n)$ of degree $L$ with amplitude $\alpha = L - \sum_{i = 0}^n \frac{L}{a_i} = L - \sum_{i = 0}^n w_i$.   
\end{example}

\begin{assumption}
    From this point onward, whenever we view $B_{a_0, \dots, a_n}$ as an $\Nat$-graded ring, we are assuming the grading defined in Example \ref{pbGrading}.
\end{assumption}

Generally, a weighted projective variety $X$ is singular and its singularities need not be easy to compute. In certain special cases, in particular when $X$ is a weighted hypersurface with nice properties, it becomes quite easy to compute $\Sing(X)$. The relevant properties are those of \textit{well-formedness} and \textit{quasismoothness}. 

\begin{definition}
    A weighted projective space $\PPP = \PPP(w_0, \dots, w_n)$ is said to be \textit{well-formed} if for each $i = 0,\dots,n$, we have $\gcd(w_0,\dots,w_{i-1},\hat{w_i},w_{i+1}, \dots , w_n) = 1$.  A weighted projective variety $X \subseteq \PPP$ is \textit{well-formed} if both $\PPP$ is well-formed and $\codim_X(X \cap \Sing(\PPP)) \geq 2$. 
\end{definition}

\begin{definition}\label{quasismoothDefinitions}
    Given a weighted projective variety $X = V_+(I) \subset \PPP(w_0, \dots, w_n)$, we can define the (affine) closed subvariety $C_X = V(I) \subseteq \aff^{n+1}$, which is called the  \textit{affine cone over $X$}; note that $C_X$ passes through the origin of $\aff^{n+1}$ and $C_X \isom \Spec( R/I )$ is an integral affine scheme. The variety $X$ is called \textit{quasismooth} if $C_X$ is nonsingular away from the origin. 
\end{definition}

\begin{definition}
    A normal $\bk$-variety $X$ is $\Rat$-factorial if for every Weil divisor $D$ there exists some $n \in \Nat^+$ such that $nD$ is Cartier. 
\end{definition}

\begin{remark}\label{qswciProperties}
    Well-formed quasismooth hypersurfaces in weighted projective space are special cases of well-formed quasismooth weighted complete intersections which have the following very nice properties: 
    \begin{enumerate}
        \item they are normal (because the affine cone $C_X$ is normal by Serre's Normality Criterion);
        \item they have at most cyclic quotient singularities \cite[p.105]{iano-fletcher_2000};
        \item they have at most rational singularities \cite[Corollary 7.4.10]{ishii2014introduction}; 
        \item they are Cohen-Macaulay \cite[Theorem 3.1A(c)]{beltramettiRobbiano};
        \item they are $\Rat$-factorial. \cite[Proposition 5.15]{kollar_mori_1998}      
    \end{enumerate}
\end{remark}

\begin{proposition}\cite[p.185 and Proposition 8]{Dimca1986}.
    Suppose $X \subseteq \PPP(w_0, \dots, w_n)$ is well-formed and quasismooth. Then $\Sing(X) = X \cap \Sing(\PPP)$. 
\end{proposition}

We make use of the following Lemma many times. 

\begin{lemma}\label{projRd}
    Suppose $R$ is $\Nat$-graded and let $d \in \Nat^+$. Then $\Proj R \isom \Proj R^{(d)}$. 
\end{lemma}

The singular points of a well-formed weighted projective space are easy to compute. First, we remark that every weighted projective space is isomorphic to a well-formed one (this is a consequence of Lemma \ref{projRd}) and so for most purposes (including ours) it suffices to identify the singular points of well-formed weighted projective spaces. They are given by the following result:

\begin{proposition}\label{singularPoints}\cite[Theorem 1.3.8]{przyjalkowski2026weighted}
Let $\PPP = \PPP(w_0, \dots, w_n) = \Proj \bk_{w_0, \dots, w_n}[x_0, \dots, x_n]$ be well-formed and let $w = \prod_{i=0}^n w_i$. For each prime factor $p$ of $w$, let
$J_p = \setspec{ j }{ 0 \le j \le n \text{ and } p \nmid w_j }$ and consider the ideal
$\qgoth_p = \sum_{j \in J_p} R x_j$. Then 
$\displaystyle \Sing \PPP = \bigcup_{ p \mid w } V_+(\qgoth_p)$, where $p$ runs over the prime factors of $w$. Equivalently, a point $P = [c_0 : \cdots : c_n] \in \PPP$ is singular
if and only if $\gcd \setspec{ w_i }{ c_i \neq 0 } > 1$.
\end{proposition}
\begin{remark}
Note that Proposition \ref{singularPoints} is false without the well-formedness assumption. Indeed, $\PPP(1,2,2) \isom \PPP^2$ is smooth, but the proposition applied to $\PPP(1,2,2)$ would imply that $\Sing(\PPP(1,2,2)) = V_+(x_0)$ which is non-empty.   
\end{remark}

\begin{example}
    Let $\PPP = \PPP(3,3,4,1)$. Then $w = 36$ and the prime factors of $w$ are $2$ and $3$. Applying Proposition \ref{singularPoints} gives $\Sing(\PPP) = \setspec{[\alpha: \beta: 0: 0]}{\alpha, \beta \text{ not both 0}} \cup \{[0:0:1:0]\}$. Thus the singular locus is a union of a projective line and a point. 
\end{example}

Another very nice property of well-formed quasismooth hypersurfaces is the fact that we can easily identify their canonical divisors. This is used very concretely in the proofs that $B_{2,3,4,12}$ and $B_{2,3,5,30}$ are rigid and in the classification of rational Pham-Brieskorn threefolds.  
\begin{theorem}\cite[Theorem 3.4.4]{dolgachev}
 Let $X \subseteq \PPP(w_0, \dots,w_n)$ be a well-formed quasismooth weighted complete intersection. Then $\omega_X \isom \omega^\circ_X \isom \OSheaf_X(\alpha)$ where $\alpha$ is the amplitude
of $X$.
\end{theorem}

We also need the following lemma.

\begin{lemma}\cite[Lemma 7.1, hypersurface case]{iano-fletcher_2000}. \footnote{Note that the last line should read $A_{-(k-\alpha)}$ instead of $A_{-k-\alpha}$ as in \cite{iano-fletcher_2000}.}
\label{FletcherCohomology}
	Let $X_f \subseteq \PPP(w_0, \dots, w_n)$ be a well-formed quasismooth hypersurface. Let $A$ be the graded ring $\bk_{w_0, \dots, w_n}[X_0, \dots, X_n] / \lb f \rb$. Then,

	\[
	H^i(X, \OSheaf_{X}(k)) \isom 
	\begin{cases}
	A_k & \text{if } i = 0\\
	0  & \text{if } 1 \leq i < \dim X \\
	A_{-(k-\alpha)} & \text{if } i = \dim X \\
	\end{cases}
	\]
	for all $k \in \Integ$. In particular, $H^0(X, \omega_X) = H^n(X, \OSheaf_X)$.  
\end{lemma}

\section{The $\cotype$ of a Pham-Brieskorn ring}
The following definitions appear in \cite{Chitayat_Daigle_2019}. In order to keep consistent notation throughout this article, our indices run from $0$ to $n$ (as opposed to from $1$ to $n$). 
	
	\begin{definitions}	\label{typeDef}
		Let $n \geq 2$ and $S = (b_0, \dots,  b_n) \in \Integ^{n+1}$.  
		\begin{itemize}
			
			\item Define\footnote{By convention, $\gcd(S)\geq 0$ and $\lcm(S)\geq 0$.}  $\gcd(S) = \gcd(b_0, \dots, b_n)$ and $\lcm(S) = \lcm(b_0,  \dots,  b_n)$.
			
			\item If $\gcd(S) = 1$, we say that $S$ is \textit{normal}.
			If $S \neq (0, \dots, 0)$ and $d = \gcd(S)$, then the tuple $S' = (\frac{b_0}{d},  \dots,   \frac{b_n}{d})$ is normal,
			and is called the \textit{normalization of $S$}.
			
			\item For each $j \in \{0,\dots,n\}$, define $S_j = (b_0, \dots , \widehat{b_j} , \dots, b_n)$.
			
			\item We define 			$\cotype(S) = |\setspec{i \in \{ 0, \dots, n \} }{ \lcm(S_i) \neq \lcm(S) }| =| \setspec{i \in \{ 0, \dots, n \} }{ b_i \nmid \lcm(S_i) }|$.    \end{itemize} 
    Note that $\cotype(S) \in \{0,1,\dots,n+1\}$ and that, if $S'$ is the normalization of $S$, then $\cotype(S) = \cotype(S')$.

	\end{definitions}

	\begin{definition}\label{ordering}
		Let $n \geq 2$.
		\begin{itemize}
			
			\item Given $S = (a_0,\dots,a_n) \in (\Nat^+)^{n+1}$ and $i \in \{0, \dots, n\}$, define $g_i(S) = \gcd(a_i, \lcm(S_i))$.
			
			\item Let $S = (a_0,\dots,a_n)$ and $S' = (a_0',\dots,a_n')$ be elements of $(\Nat^+)^{n+1}$ and let $i \in \{0, \dots, n\}$.
			We write $S \leq^i S'$ if and only if
			$$
			S_i = S'_i \quad \text{and} \quad g_i(S') \mid a_i \mid a_i'.
			$$
			We write $S <^i S'$ if and only if $S \leq^i S'$ and $S \neq S'$. Note that the relation $\leq^i$ is a partial order on $(\Nat^+)^{n+1}$.
		\end{itemize}
	\end{definition}

    \begin{proposition}\cite[Proposition 5.2]{Chitayat_Daigle_2019}\label{cotypeZero}
        Let $n \geq 2$, $S = (a_0, \dots , a_n)$ and $S' = (a_0', \dots, a_n')$. Assume that $i \in \{0,1,\dots, n\}$ is such that $S \leq^i S'$. Then there is a graded isomorphism $B_S \isom B_{S'}^{(k)}$ where $k = a_i' / a_i$. Consequently, $\Proj B_S \isom \Proj B_{S'}$. 
    \end{proposition}

    \begin{corollary}\cite[Section 4.5]{chitayat2025rationality}\label{ProjIsoCorollary}
        Given $B_{a_0, \dots, a_n}$, there exists some $B_{b_0, \dots, b_n}$ of cotype 0 such that $\Proj B_{a_0, \dots, a_n} \isom \Proj B_{b_0, \dots, b_n}$. 
    \end{corollary}
    
    The proofs of Proposition \ref{cotypeZero} and Corollary \ref{ProjIsoCorollary} are not difficult. Rather than give them, let us demonstrate the method with an example. The proofs are simply the generalization of this method to an arbitrary tuple. 

    \begin{example}
        Let $B = B_{10,3,3,4} = \bk[x_0, x_1, x_2, x_3] / \lb x_0^{10} + x_1^3 + x_2^3 + x_3^4 \rb$, so that with the notation of \ref{typeDef}, we have $S' = (10,3,3,4)$ and $S'_0 = (3,3,4)$. We have $\lcm(S') = 60 \neq 12  =\lcm(S'_0)$. Set $S = (g_0(S'), 3,3,4) = (2,3,3,4)$. Then $S <^0 S'$ and furthermore $\cotype(S) = 1 < 2 = \cotype(S')$. We claim that $ B_{2,3,3,4} \isom B_{10,3,3,4}^{(5)}$. Indeed, it is easily checked that 
        \begin{align*}
            (B_{10,3,3,4})^{(5)} &= \left(\bk_{6,20,20,15}[x_0, x_1, x_2, x_3] / \lb x_0^{10} + x_1^3 + x_2^3 + x_3^4 \rb\right)^{(5)}\\
            &= \bk_{30,20,20,15}[x_0^5, x_1, x_2, x_3] / \lb x_0^{10} + x_1^3 + x_2^3 + x_3^4 \rb \\
            & \isom \bk_{30,20,20,15}[y_0, x_1, x_2,x_3] / \lb y_0^2 + x_1^3 + x_2^3 + x_3^4 \rb \\ 
            &\isom \bk_{6,4,4,3}[y_0, x_1, x_2,x_3] / \lb y_0^2 + x_1^3 + x_2^3 + x_3^4 \rb = B_{2,3,3,4}.
        \end{align*}
        The exact same process shows that $(2,3,3,2) <^3 (2,3,3,4)$, where $\cotype(2,3,3,2) = 0 < 1 = \cotype(2,3,3,4)$. 
    \end{example}

For Pham-Brieskorn rings $B = B_{a_0, \dots, a_n}$, the notion of $\cotype$ is useful as it allows us to characterize exactly when $\Proj B_{a_0, \dots, a_n}$ is well-formed.

\begin{proposition}\label{PBsatCodim1}\cite[Proposition 2.14]{chitayat2025rigid}
	Let $n \geq 2$. Let $f = x_0^{a_0} + \dots  + x_n^{a_n} \in \bk[x_0, \dots , x_n]$, let $B = B_{a_0, \dots, a_n}$ and let $X_f = \Proj B$. The following are equivalent:
		
	\begin{enumerate}[\rm(a)]
        \item $\cotype(a_0, \dots, a_n) = 0$;            
        \item $X_f$ is quasi-smooth and well-formed.
	\end{enumerate}
\end{proposition}

\section{Rigidity}\label{rigiditySection}
\begin{definition}
    Let $R$ be a ring. A derivation $D: R \to R$ is \textit{locally nilpotent} if for every $r \in R$, there exists some $n$ such that $D^n(r) = 0$. A ring $R$ is said to be \textit{rigid} if the only locally nilpotent derivation is the zero derivation. We denote the set of locally nilpotent derivations of $R$ by $\lnd(R)$. 
\end{definition}

In this section we discuss the following conjecture. 

\begin{conjecture}\label{PBConjecture}\cite{Kali-Zaid_2000}
    Let $B = B_{\bk; a_0, \dots, a_n}$ be a Pham-Brieskorn ring and assume $\Rat(i) \subseteq \bk$. Then $B_{\bk; a_0, \dots, a_n}$ is non-rigid if and only if $a_i = 1$ for some $i$ or $a_i = a_j = 2$ for some $i \neq j$. 
\end{conjecture}

\begin{remark}\label{easyPart}
    The $(\Leftarrow)$ direction of Conjecture \ref{PBConjecture} is easy to prove. If $a_i = 1$ for some $i$, then $B_{a_0, \dots, a_n}$ is isomorphic to the polynomial ring $\bk[x_0, \dots, x_{i-1}, \hat{x_i}, x_{i+1}, \dots, x_n]$ which is non-rigid. If (without loss of generality) $a_0 = a_1 = 2$, then we can observe that $x_0^2 + x_1^2 = (x_0 + ix_1)(x_0-ix_1)$ (since $\Rat(i) \subseteq \bk$) and deduce that $B_{a_0, \dots, a_n} \isom \bk[u,v, x_2, \dots, x_n] / \lb uv + x_2^{a_2} + \dots + x_n^{a_n} \rb$, the latter admitting the non-zero locally nilpotent derivations $u \frac{\partial}{\partial {x_i}} - a_ix_i^{a_i -1} \frac{\partial}{\partial v}$ for each $i \geq 2$. This proves $(\Leftarrow)$.  
\end{remark}

\begin{remark}
    We will see that the ring $\Reals[x_0, x_1, x_2] / \lb x_0^2 + x_1^2 + x_2^2 \rb$ is rigid, showing that the assumption $\bk \supseteq \Rat(i)$ in Conjecture \ref{PBConjecture} is necessary for the $(\Leftarrow)$ direction to hold. One proof of this claim is given in \cite[Theorem 9.27]{freudenburg2017algebraic}. We will give a different proof, due to Daigle. 
\end{remark}

\begin{theorem}\cite[Introduction]{russell1970forms} \label{extensionLine}
Let $K / \bk$ be a field extension of characteristic zero.
If $A$ is a finitely generated $\bk$-algebra such that $K \otimes_\bk A = K^{[1]}$, then $A = \bk^{[1]}$.
\end{theorem}

\begin{lemma}\label{rationalPointsAreDense}
    Let $\bk$ be an infinite field and let $B$ be an affine $\bk$-domain with
    $\Frac B = \bk^{(n)}$ for some $n \geq 0$. Then the set of $\bk$-rational points of
    $\Spec B$ is dense in $\Spec B$.
\end{lemma}
\begin{proof}
    Suppose $C \subseteq \Spec B$ is closed and contains every $\bk$-rational point of $\Spec B$. Write $C = V(I)$ for some ideal $I \subseteq B$. It suffices to show
    $I = 0$.

    Suppose $I \neq 0$ and pick $f \in I$ with $f \neq 0$. Since every $\bk$-rational point
    $\mgoth$ is in $C = V(I)$, $f \in \mgoth$ for every $\bk$-rational $\mgoth$, and hence
    \begin{equation}\label{finM}
    \text{$D(f)$ contains no $\bk$-rational point.} 
    \end{equation}

    Write $B = \bk[b_1, \dots, b_m]$ and $\Frac B = \bk(t_1, \dots, t_n)$, and choose
    $p_i, q_i \in \bk[t_1, \dots, t_n]$ with $q_i \neq 0$ and $b_i = p_i / q_i$. Setting
    $q = q_1 \cdots q_m \neq 0$, every generator $b_i$ lies in
    $\bk[t_1, \dots, t_n]_q$, and therefore
    \[
        B \subseteq \bk[t_1, \dots, t_n]_q .
    \]
    Consequently $f = P / q^N$ for some $P \in \bk[t_1, \dots, t_n]$ and some $N \geq 0$,
    and $P \neq 0$ because $f \neq 0$.

    Since $qP$ is a nonzero polynomial and $\bk$ is infinite, there exists
    $a \in \bk^n$ with $q(a) \neq 0$ and $P(a) \neq 0$. As $q(a) \neq 0$, evaluation at
    $a$ defines a $\bk$-algebra homomorphism
    $\varepsilon_a \colon \bk[t_1, \dots, t_n]_q \to \bk$. Put
    $\varphi = \varepsilon_a |_B$ and $\mgoth = \ker \varphi$. The image of
    $\varphi$ is a subring of $\bk$ containing $\bk$, hence equals $\bk$, so
    $B / \mgoth \isom \bk$ and $\mgoth$ is a $\bk$-rational point of $\Spec B$. Since $\varphi(f) = \frac{P(a)}{q(a)^N} \neq 0$, $f \notin \mgoth$ and so $\mgoth \in D(f)$. This contradicts \eqref{finM}.
\end{proof}

\begin{proposition}
    The ring $\Reals[x_0, x_1, x_2] / \lb x_0^2 + x_1^2 + x_2^2 \rb$ is rigid. 
\end{proposition}
\begin{proof}
    By contradiction, assume $B = \Reals[x_0,x_1,x_2] / (x_0^2 + x_1^2 + x_2^2)$ is not rigid. Let $D \in \lnd(B) \setminus \{0\}$ and $A=\ker(D)$. By Remark \ref{easyPart}, $\bar B = \Comp \otimes_\Reals B$ is isomorphic to $\Comp[u,v,w]/ \lb uv-w^2 \rb$ and by \cite[Lemma 2.8]{daigle2004locallyDanielewski}, every element of $\lnd(\bar{B}) \setminus \{0\}$ has a kernel isomorphic to $\Comp^{[1]}$. 
    
Applying the exact functor $\Comp \otimes_\Reals(\underline{\ \ })$ to the exact sequence $0 \to A \to B \xrightarrow{D} B$
shows that $0 \to \Comp \otimes_\Reals A \to \bar B \xrightarrow{\bar D} \bar B$ is exact, so
$\Comp \otimes_\Reals A = \ker \bar D = \Comp^{[1]}$.
By Theorem \ref{extensionLine}, $A = \Reals^{[1]}$, so $\Frac B = (\Frac A)^{(1)} = \Reals^{(2)}$ (using \cite[Principle 11]{freudenburg2017algebraic}). Since $\Spec B$ has only one $\Reals$-rational point (the point $\lb x_0,x_1,x_2 \rb$), this contradicts Lemma \ref{rationalPointsAreDense}.
\end{proof}

\subsection{The $2$-dimensional case.}

Conjecture \ref{PBConjecture} was shown to be true for $n = 2$ in \cite[Lemma 4]{Kali-Zaid_2000} although we will outline the proof given in \cite{freudenburg2017algebraic}. Items \ref{freeRelativelyPrime} to \ref{qeLemma} are the four main ingredients that go into this proof of the 2-dimensional case of Conjecture \ref{PBConjecture}.  


\begin{lemma}\cite[Lemma 2.46]{freudenburg2017algebraic}\label{freeRelativelyPrime}
Let $B$ be a commutative $\bk$-domain and $R \subseteq B$ a subalgebra such that
$B$ is a free $R$-module. Given $x, y \in R$, if $x$ and $y$ are relatively prime in $R$, then $x$ and
$y$ are relatively prime in $B$.
\end{lemma}

\begin{theorem}\label{ABCTheorem}{\rm [ABC Theorem]}\cite[Theorem 6.1]{freudenburg2013}
    Suppose $x,y,z \in B$ are pairwise relatively prime and satisfy $x^a + y^b + z^c = 0$ where $a,b,c \geq 2$ and $\frac{1}{a} + \frac{1}{b} + \frac{1}{c} \leq 1$. Then every $D \in \lnd(B)$ satisfies $D(x) = D(y) = D(z) = 0$. 
\end{theorem}

\begin{theorem}\label{ABTheorem}{\rm [AB Theorem]}\cite[Theorem 2.47]{freudenburg2017algebraic}
    Let $D  \in \lnd(B) \setminus \{0\}$. Suppose $u, v \in \ker D$ and $x, y \in B$ are nonzero and that $a, b \geq 2$. Assume $u x^a + vy^b \neq 0$.  
    \begin{enumerate}[\rm(a)]
        \item If $D(ux^a + vy^b) = 0$ then $D(x) = D(y) = 0$.
        \item If $D^2(ux^a + vy^b) = 0$ and $a$ and $b$ are not both 2, then $D(x) = D(y) = 0$.
    \end{enumerate}
\end{theorem}
\begin{lemma}\label{qeLemma}\cite[Corollary 3.2]{freudenburg2013}
    Suppose that $R$ is a $\Integ$-graded affine $\bk$-domain, $f \in R$ is homogeneous, $\deg f  \neq 0$ and $n \geq 2$ is an integer relatively prime to $\deg f$. Let $x$ be an indeterminate over $R$ and assume $B = R[x]  /\lb x^n - f \rb $ is a domain. Then, the following conditions are equivalent:
    \begin{enumerate}[\rm(a)]
    \item $D^2(f) \neq 0$ for every nonzero $D \in \lnd(R)$;
    \item $B$ is rigid.
    \end{enumerate}
\end{lemma}

\begin{proposition}\label{PB2d}
    Conjecture \ref{PBConjecture} is true for $n = 2$. 
\end{proposition}
\begin{proof}
    By Remark \ref{easyPart}, it suffices to prove $(\Rightarrow)$, which we do by contrapositive.  Observe that $B_{a_0, a_1, a_2}$ is a free $\bk[x_0, x_1]$-module (of rank $a_2$). Clearly $x_0$ and $x_1$ are relatively prime in the polynomial ring $\bk[x_0, x_1]$ and so Lemma \ref{freeRelativelyPrime} gives that they are relatively prime in $B_{a_0, a_1, a_2}$. The same argument shows that $x_0, x_1, x_2$ are pairwise relatively prime. If $\frac{1}{a_0} + \frac{1}{a_1} + \frac{1}{a_2} \leq 1$, then Theorem \ref{ABCTheorem} implies that $D(x_0) = D(x_1) = D(x_2)= 0$, so we may assume henceforth that $\frac{1}{a_0} + \frac{1}{a_1} + \frac{1}{a_2} > 1$. By assumption, no $a_i$ equals 1 and no pair of distinct $a_i$, $a_j$ equal $2$. So without loss of generality we may assume $(a_0, a_1, a_2) \in \{(2,3,3), (2,3,4), (2,3,5)\}$. 

    Consider first the $(2,3,3)$ case. Let $R = \bk[x_1,x_2]$ and equip $R$ with the standard grading. Then $f = -x_1^3 - x_2^3$ is homogeneous of degree 3. By Theorem \ref{ABTheorem} (b) (setting $u = v = 1$, $a = b = 3$), it follows that condition (a) of Lemma \ref{qeLemma} is satisfied when applied to $R = \bk[x_1, x_2]$. Since $B_{2, 3, 3}  =R[x_0] / \lb x_0^2 - f \rb$, it follows from that Lemma that $B_{2,3,3}$ is rigid.   

    The proof of the $(2,3,4)$ case (resp. $(2,3,5)$ case) is similar, where we let $R = \bk_{2,1}[x_0, x_2]$ (resp. $R = \bk_{3,2}[x_0, x_1]$) and argue in the same way.    
\end{proof}

\subsection{Two reductions of Conjecture \ref{PBConjecture}.}

Many partial results about the higher dimensional cases of Conjecture \ref{PBConjecture} were obtained in the interim (see in particular \cite{freudenburg2013}, \cite{DFM2017}, \cite{affineCones}, \cite{cheltsov_park_won_2016}, \cite{CylindersInDelPezzoSurfaces}, \cite{Chitayat_Daigle_2019},\cite{LiuSunExtensions}). These results were generally obtained as applications or corollaries of otherwise useful results and could not be combined to give a complete proof. Here we will discuss two reductions of Conjecture \ref{PBConjecture} and give the ideas behind the proof of the 3-dimensional case. The complete proof of the $n = 3$ case of Conjecture \ref{PBConjecture} for $\bk = \Comp$ can be found in \cite{chitayat2025rigid}, so here we only give the main elements. Note also that Remark \ref{LefschetzRigidity} explains the generalization to arbitrary $\bk$ containing $\Rat(i)$. One of the interesting elements of the proof of the $n = 3$ case of Conjecture \ref{PBConjecture} is the fact that it requires techniques from both the theory of locally nilpotent derivations and from the theory of birational geometry of surfaces. I don't know how one might go about replacing the algebraic part of the proof with a geometric argument, and vice versa. 

The first algebraic reduction used in the proof is a result from my thesis and consists of the following reduction to two simpler cases. This reduction is valid in all dimensions. We have:

\begin{theorem}\cite[Theorem 1.3.12]{ChitayatThesis}\label{reduction}
    To prove Conjecture \ref{PBConjecture} for $n = k$, it suffices to prove:
    \begin{enumerate}[\rm(1)]
    \item Conjecture \ref{PBConjecture} is true for $n = k-1$,
    \item Conjecture \ref{PBConjecture} is true when $n = k$ and $\cotype(a_0, \dots, a_n) = 0$.
    \end{enumerate}
\end{theorem}

The next step is to apply a beautiful result of Kishimoto, Prokhorov and Zaidenberg, connecting the existence of non-zero locally nilpotent derivations of an $\Nat$-graded ring $B$ with the existence of certain special types of cylinders on $\Proj B$. Theorem \ref{PBCanonical2} below is an application of their original result, which can be found as \cite[Theorem 0.6]{KishimotoProkhorovZaidenberg}. Note that once again, this result is valid in any dimension. Before stating it however, we need a definition. 

We recall that a $\Rat$-divisor on a normal projective variety $X$ is an element of $\Rat \otimes_\Integ \Div(X)$. We denote the set of $\Rat$-divisors by $\Div(X, \Rat)$. A $\Rat$-divisor $D$ is ample if there exists some $n \in \Nat^+$ such that $\OSheaf_X(nD)$ is a very ample invertible sheaf. It is $\Rat$-Cartier if there exists some $n$ such that $nD$ is a Cartier divisor.   

\begin{definition}
    Let $X$ be a variety over $\bk$. A \textit{cylinder in $X$} is an open subset $U$ such that $U \isom Z \times_\bk \Aff^1_\bk$ for some affine variety $Z$. 
\end{definition}

\begin{definition}
    Let $H$ be an ample $\Rat$-divisor on a projective normal $\bk$-variety $X$. Let
$U \subseteq X$ be a cylinder of $X$. Then, the cylinder $U \subseteq  X$ is \textit{$H$-polar} if $U = X \setminus \Supp(D)$ for
some effective $\Rat$-Cartier $\Rat$-divisor $D \in \Div(X, \Rat)$ such that $D \sim sH$ where $s \in \Rat^+$.
\end{definition}

The following combines \cite[Theorem 4.1.14]{ChitayatThesis} (the aforementioned application of \cite[Theorem 0.6]{KishimotoProkhorovZaidenberg}) with \cite[Proposition 2.3.2]{chitayat2025rigid} which essentially says that a normal variety with mild enough singularities and pseudoeffective canonical divisor cannot contain a cylinder. 

\begin{theorem} \label{PBCanonical2}
	Let $n \geq 2$ and consider $(a_0,\dots, a_n) \in (\Nat^+)^{n+1}$ of cotype $0$.
	Let $B = B_{a_0,\dots, a_n}$, let $X = \Proj(B)$, and let $\alpha$ be the amplitude of $X$. Then,
	\begin{enumerate}[\rm(a)]
		
		\item $\omega_X \isom \OSheaf_X(\alpha)$.
		\item If $\alpha \geq 0$, then $K_X$ is pseudoeffective and $B$ is rigid.
		\item Assume that $\alpha < 0$. Then $-K_X$ is ample and the following are equivalent:
		\begin{enumerate}[\rm(i)]
			\item $B$ is not rigid;			
			\item there exists a $(-K_X)$-polar cylinder of $X$.
		\end{enumerate}
	\end{enumerate}
\end{theorem}
\begin{remark}\label{alphaEquivalence}
    Let $B = B_{a_0,\dots, a_n}$ and let $X = \Proj B$. Then it is easy to check that $\alpha < 0$ if and only if $\sum_{i = 0}^n \frac{1}{a_i} > 1$. 
\end{remark}

\begin{nothing}
Normal varieties with mild singularities (say klt singularities) such that $-K_X$ is $\Rat$-Cartier and ample are called \textit{Fano varieties}. If $B = B_{a_0, \dots, a_n}$, $X = \Proj B$ and $-K_X$ is ample, then $X$ is indeed a Fano variety.  By combining Theorem \ref{reduction} with Theorem \ref{PBCanonical2}, an easy induction argument shows that in order to prove Conjecture \ref{PBConjecture} for every $n$, it suffices to prove it in the special case where $\cotype(a_0, \dots, a_n) = 0$ and $\Proj B_{a_0, \dots, a_n}$ is a Fano variety.
\end{nothing}

\subsection{The $3$-dimensional case}

Because we know that the $n = 2$ case of Conjecture \ref{PBConjecture} is true by Proposition \ref{PB2d}, to prove the $n = 3$ case of Conjecture \ref{PBConjecture} it suffices to prove it in the special case where $\cotype(a_0, a_1, a_2, a_3) = 0$ and $\Proj B_{a_0, a_1, a_2, a_3}$ has ample anti-canonical divisor. The nice thing about the 3-dimensional case is that (up to a permutation) there are only 8 tuples $(a_0, a_1, a_2, a_3)$ that satisfy the three conditions
\begin{enumerate}[\rm(i)]
    \item no $a_i = 1$ and no two distinct $i,j$ satisfy $a_i = a_j = 2$,
    \item $\cotype(a_0, a_1,a_2,a_3) = 0$,
    \item $X = \Proj B_{a_0, a_1, a_2, a_3}$ has ample anti-canonical divisor (equivalently $\alpha < 0$),
\end{enumerate}
and it is an exercise in elementary number theory to show that this list of 8 tuples consists of the following elements:
\begin{itemize}
    \item $(2,3,3,6)$
    \item $(2,3,6,6)$
    \item $(2,4,4,4)$
    \item $(3,3,3,3)$
    \item $(3,3,4,4)$
    \item $(3,3,5,5)$
    \item $(2,3,4,12)$
    \item $(2,3,5,30)$.
\end{itemize}

By Theorem \ref{PBCanonical2}, it suffices to show that for each of the above tuples, $X = \Proj B_{a_0, a_1, a_2, a_3}$ doesn't contain a $-K_X$-polar cylinder. After classifying the singularity types of these 8 surfaces, it turns out that the first 6 tuples in the list are members of families that were already known to not contain $-K_X$-polar cylinders. (See \cite[Sections 4.5 and 4.6]{ChitayatThesis} for details.) The $(2,3,3,6)$, $(2,3,6,6)$,$(2,4,4,4)$ and $(3,3,3,3)$ cases can be shown to not have $-K_X$-polar cylinders by applying Theorem \ref{thm:CPWantiCanonical} below.

\begin{theorem}$($\cite[Theorem 1.5]{cheltsov_park_won_2016}$)$\label{thm:CPWantiCanonical}
Let $S$ be a del Pezzo surface of degree $d$ with at most Du Val singularities. The surface $S$ does not admit a $-K_{S}$-polar cylinder if and only if one of the following hold:
    \begin{enumerate}[\rm(1)]
		\item $d = 1$ and $S$ allows only singular points of types $A_1, A_2, A_3, D_4$ if any;
		\item $d = 2$ and $S$ allows only singular points of type $A_1$ if any;
		\item $d = 3$ and $S$ allows no singular point.
\end{enumerate}
\end{theorem}

The $(3,3,4,4)$ and $(3,3,5,5)$ cases are a consequence of \cite[Lemmas 4.1 and 5.1]{Cheltsov2010}. Those results show that the log canonical thresholds of $\Proj B_{3,3,4,4}$ and $\Proj B_{3,3,5,5}$ are at least 1, and that condition together with \cite[Theorem 1.2.6]{cheltsov2020cylinders} implies that both of those surfaces cannot contain an anti-canonical polar cylinder. 

In \cite{chitayat2025rigid}, we show that the same is true for the $(2,3,4,12)$ and $(2,3,5,30)$ cases. The proofs that $\Proj B_{2,3,5,30}$ and $\Proj B_{2,3,4,12}$ do not contain $-K_X$-polar cylinders are given by Propositions 3.2.7 and 3.3.9 in \cite{chitayat2025rigid}. The method in the $(2,3,5,30)$ case (resp. $(2,3,4,12)$ case) is to assume that $X$ contains a $-K_X$-polar cylinder and to derive a contradiction by showing that after a sequence of blowing up and blowing down operations, we can produce an anti-canonical polar cylinder in a smooth del Pezzo surface of degree 1 (resp. degree 2), which contradicts the classification given in Theorem \ref{thm:CPWantiCanonical}.

In conclusion we have:
\begin{theorem}\label{RigidityTheorem}
Assume $\Rat(i) \subseteq \bk$ and let $B = B_{\bk; a_0, a_1, a_2, a_3}$ be a 3-dimensional Pham-Brieskorn ring. Then $B_{\bk; a_0, a_1, a_2, a_3}$ is non-rigid if and only if $a_i = 1$ for some $i$ or $a_i = a_j = 2$ for some $i \neq j$.
\end{theorem}

\begin{remark}\label{LefschetzRigidity}
    Note that Theorem \ref{RigidityTheorem} is stated for any $\bk$ containing $\Rat(i)$ but the original proof is given over $\Comp$. This more general case can be obtained by applying the Lefschetz type argument demonstrated in Corollary \ref{generalCancellation}. 
\end{remark}

\begin{problem}
    Extend the classification of which del Pezzo surfaces $S$ contain $-K_S$-polar cylinders to include surfaces with quotient singularities. 
\end{problem}

\begin{example}
    It is worth noting that Theorem \ref{thm:CPWantiCanonical} implies the existence of
    rational surfaces that contain cylinders but do not contain $-K_X$-polar cylinders.
    For example, let $X = \Proj B_{\Comp; 3, 3, 3, 3}$. Then $X$ is a smooth cubic surface in $\PPP^3$ and hence is a del Pezzo surface of degree $K_X^2 = 3$. By
    Theorem \ref{thm:CPWantiCanonical}, $X$ contains no $-K_X$-polar cylinder. We exhibit a cylinder in $X$.

    Being a smooth del Pezzo surface of degree $3$, $X$ is obtained from $\PPP^2$ by
    blowing up $6$ points $p_1, \dots, p_6$ in general position (no three on a line and
    not all six on a conic), giving a sequence
    \[
        X = X_6 \overset{\pi_6}{\longrightarrow} X_5 \overset{\pi_5}{\longrightarrow}
        \dots \overset{\pi_2}{\longrightarrow} X_1 = \mathbb{F}_1
        \overset{\pi_1}{\longrightarrow} \PPP^2 ,
    \]
    where each $\pi_i$ is the blowup at $p_i \in X_{i-1}$. Write $E_i$ for the total
    transform in $X$ of the exceptional curve of $\pi_i$. Since no three of the $p_i$ are
    collinear, under the ruling $\mathbb{F}_1 \to \PPP^1$ induced by the pencil of lines through $p_1$, the points $p_2, \dots, p_6$ lie on distinct fibers. Let
    $F_i \subseteq \mathbb{F}_1$ be the fiber containing $p_i$, and let $\tilde{F}_i \subseteq X$ denote its
    strict transform. Then 
    \[
        U = X \setminus \left( \bigcup_{i=1}^{6} E_i \cup \bigcup_{i=2}^{6} \tilde{F}_i \right)
    \]
    is a cylinder in $X$.
\end{example}

\subsection{The higher dimensional cases}

When $n \geq 4$, $B = B_{\bk; a_0, \dots, a_n}$ is a UFD by Proposition \ref{highDimUFD} and it can be shown that $\CaCl(\Proj B) = \Integ$. In view of the following (relatively straightforward) proposition, we obtain that every cylinder in $X = \Proj B$ is an $H$-polar cylinder for some ample $\Rat$-divisor $H$.

\begin{proposition}\cite[Proposition 4.9.5]{ChitayatThesis}
Let $X$ be a projective normal $\Rat$-factorial variety and suppose that
$\CaCl(X) \isom \Integ$. Let $H$ be any ample $\Rat$-divisor on $X$. Then every cylinder $U$ in X is an $H$-polar cylinder. 
\end{proposition}

Thus, to prove Conjecture \ref{PBConjecture} for $n \geq 4$, it suffices to solve the following (more general) problem in the special case of Pham-Brieskorn hypersurfaces. 

\begin{problem}
    Let $X_f \subseteq \PPP(w_0, \dots, w_n)$ be a well-formed quasismooth weighted hypersurface of dimension at least 3 that is also a Fano variety. Describe when $X_f$ contains a cylinder. 
\end{problem}

The study of when a Fano variety of dimension at least 3 admits a cylinder is generally a difficult problem, and many papers have been written on the topic. For an excellent survey on the topic, see \cite{cheltsov2020cylinders}.


\section{Rationality}

In this section, we consider the following natural question; the $n = 3$ case of Question \ref{PBRationalQuestion} was asked to me by Rajendra V. Gurjar during my Ph.D. oral exam. 

\begin{question}\label{PBRationalQuestion}
    Let $\bk$ be an algebraically closed field of characteristic zero. For which tuples $(a_0, \dots, a_n)$ is $B_{\bk; a_0, \dots, a_n}$ rational?
\end{question}
Before addressing the above, we begin with a general Lemma which shows how with rationality questions in characteristic zero, it often (but not always) suffices to consider the special case $\bk = \Comp$. The spirit of the reduction is similar to the Lefschetz principle argument described before Corollary \ref{generalCancellation}.  

\begin{lemma}\label{lem:descent}
Let $\bk$ be an algebraically closed field and let $K/\bk$ be a field extension.
Suppose $X$ is an integral $\bk$-variety. Then $X_K = X \times_\bk \Spec K$ is $K$-rational if and
only if $X$ is $\bk$-rational.
\end{lemma}

\begin{proof}
The $(\Leftarrow)$ implication is clear. For $(\Rightarrow)$, let $n=\dim X$.
Since $\bk$ is algebraically closed, $X$ is geometrically integral, so $X_K$ is
integral of dimension $n$. Since there exists a birational map
$X_K\dashrightarrow\mathbb{A}^n_K$, there exists an isomorphism
$\varphi:(R_0\otimes_{\bk}K)_g\to K[t_1,\dots,t_n]_h$ for some affine open
$\Spec R_0\subseteq X$, some nonzero $g\in R_0\otimes_\bk K$ and some nonzero
$h\in K[t_1,\dots,t_n]$. Let $\psi$ denote the inverse of $\varphi$.

Fix a $\bk$-basis of $R_0$ and $\bk$-algebra generators $r_1,\dots,r_m$ of
$R_0$, so that $(R_0\otimes_\bk K)_g$ is generated as a $K$-algebra by
$y_0=1/g$ and $y_i=r_i\otimes1$ $(1\le i\le m)$, and every element of
$(R_0\otimes_\bk K)_g$ has well-defined coefficients in $K$ once written with
denominator a power of $g$. Let $A\subseteq K$ be the $\bk$-subalgebra
generated by the coefficients of $g$ and $h$, of $\varphi(y_0),\dots,
\varphi(y_m)$, of $\psi(t_1),\dots,\psi(t_n)$, and of $\psi(h)^{-1}$. Then $A$
is a finitely generated $\bk$-algebra, $g\in R_0\otimes_\bk A$,
$h\in A[t_1,\dots,t_n]$, and $\varphi$ and $\psi$ restrict to $A$-algebra
homomorphisms
\[
\varphi_0:(R_0\otimes_\bk A)_g\to A[t_1,\dots,t_n]_h,
\qquad
\psi_0:A[t_1,\dots,t_n]_h\to (R_0\otimes_\bk A)_g. 
\]
Since $R_0$ is a free $\bk$-module and $A \subseteq K$, 
$R_0\otimes_\bk A\subseteq R_0\otimes_\bk K$. Localization is exact, so the
vertical arrows in
\[
\begin{tikzcd}[column sep=large]
(R_0\otimes_\bk A)_g \arrow[r,"\varphi_0"] \arrow[d,hook, "i"]
  & A[t_1,\dots,t_n]_h \arrow[r,"\psi_0"] \arrow[d,hook]
  & (R_0\otimes_\bk A)_g \arrow[d,hook, "i"] \\
(R_0\otimes_\bk K)_g \arrow[r,"\varphi"']
  & K[t_1,\dots,t_n]_h \arrow[r,"\psi"']
  & (R_0\otimes_\bk K)_g
\end{tikzcd}
\]
are injective. Since $\psi \circ \varphi = \Id$, $\psi_0 \circ \varphi_0 = \Id$. The analogous diagram gives $\varphi_0 \circ \psi_0 = \Id$, so $\varphi_0$ and $\psi_0$ are isomorphisms.  

Write $g=\sum_\lambda e_\lambda\otimes a_\lambda$ in the chosen basis of $R_0$ over $\bk$. Fix
$a_{\lambda_0}\neq 0$; let $b\in A$ be a nonzero coefficient of $h$. Replacing
$A$ by $A[a_{\lambda_0}^{-1},b^{-1}]$ if necessary, we may assume $a_{\lambda_0}$ and $b$ are units of $A$. Since $\bk = \bar{\bk}$, every maximal ideal $\mathfrak{m} \lhd A$ satisfies $A/\mathfrak{m}=\bk$ (by Hilbert's Nullstellensatz). Let $\lambda:A\to A / \mgoth \isom \bk$ be the quotient map and let
$g_\lambda\in R_0$, $h_\lambda\in\bk[t_1,\dots,t_n]$ be the images of $g$ and
$h$; these are nonzero because $\lambda(a_{\lambda_0})\ne0$ and
$\lambda(b)\ne0$. Applying $-\otimes_{A,\lambda}\bk$ to $\varphi_0$, and using
that base change commutes with adjoining variables and with localization, gives
an isomorphism of $\bk$-algebras $(R_0)_{g_\lambda}\cong
\bk[t_1,\dots,t_n]_{h_\lambda}$. So 
$\Frac R_0\cong\bk(t_1,\dots,t_n)$. Since $\Spec R_0$ was an affine open in $X$, $X$ is $\bk$-rational.
\end{proof}

\begin{corollary}\label{PBdescent}
    Let $\bk$ be an algebraically closed field of characteristic zero. Then $\Proj B_{\bk; a_0, \dots, a_n}$ (resp. $\Spec B_{\bk; a_0, \dots, a_n}$) is $\bk$-rational if and only if $\Proj B_{\Comp; a_0, \dots, a_n}$ (resp. $\Spec B_{\Comp; a_0, \dots, a_n}$) is $\Comp$-rational.
\end{corollary}
\begin{proof}
    Let $\kappa$ be the unique algebraic closure of $\Rat$ contained in $\bk$. Then note that we can embed $\kappa$ in $\Comp$ and so by a slight abuse of notation, we may also view $\kappa$ as a subfield of $\Comp$ after a suitable choice of embedding. Let $X_\bk = \Proj B_{\bk; a_0, \dots, a_n}$ and let $X_\Comp = \Proj B_{\Comp; a_0, \dots, a_n}$. Then by Lemma \ref{lem:descent} together with a standard base change, we have 
    $$X_\bk \text{ is $\bk$-rational} \iff X_\kappa \text{ is $\kappa$-rational} \iff X_\Comp \text{ is $\Comp$-rational.}$$
    The same argument holds for $\Spec$. 
\end{proof}

\begin{nothing}
    The $n = 1$ case of Question \ref{PBRationalQuestion} only makes sense if we assume $B_{\bk; a_0, a_1}$ is a domain. Furthermore, we saw in \ref{nequals2} that for $\bk = \bar{\bk}$, $B_{\bk; a_0, a_1}$ is rational whenever it is a domain, so this settles the $n = 1$ case. 
\end{nothing}

Given any $(a_0, \dots, a_n)$ where $n = 2$ or $n=3$ and $\bk = \bar{\bk}$, it is always possible to determine whether or not $B_{a_0, \dots, a_n}$ is rational. For the $n = 2$ case, there are numerical criteria based on the values of $(a_0, a_1, a_2)$. For the $n = 3$ case, there is an easy algorithm that determines whether or not $B_{a_0, a_1, a_2, a_3}$ is rational. The proof relies on the following easy application of Luroth's and Castelnuovo's theorems.

\begin{proposition}\label{CastelnuovoProp}
Let $\bk$ be a field and let $B=\bigoplus_{i\geq 0}B_i$ be an $\mathbb{N}$-graded finitely generated $\bk$-domain with $B_0=\bk$ and $B \neq \bk$. Assume $\dim B \leq 3$ and let $X = \Proj B$. \begin{enumerate}[\rm(a)]
    \item Then $K(\Spec B) = \operatorname{\Frac}(B)=K(X)^{(1)}$. 
    \item If $\bk$ is algebraically closed of characteristic zero, then $\Proj B$ is rational if and only if $\Spec B$ is rational. 
    \end{enumerate}
\end{proposition}

\begin{proof}
Let $S$ be the multiplicative set of nonzero homogeneous elements of $B$.
Then $B_S$ is a $\mathbb{Z}$-graded ring in which every nonzero homogeneous
element is a unit. The support of $B_S$, denoted by 
$\Supp(B_S) = \setspec{e \in \Integ}{(B_S)_e \neq 0}$ 
satisfies $\Supp(B_S) = e\Integ \subseteq \Integ$ for some $e \geq 1$ (since $B \neq B_0$). Choose some homogeneous $t\in B_S$ of degree $e$. We claim $B_S=(B_S)_0[t,t^{-1}]$. The $(\supseteq)$ direction is clear. For the converse, suppose $\frac{b}{s} \in B_S$ where $\frac{b}{s}$ is homogeneous of degree $me$ where $m \in \Integ$. Then $\frac{b}{st^{m}}t^{m} \in (B_S)_0[t,t^{-1}]$. Recall that by definition $(B_S)_0=K(\Proj B)$ and so $K(\Spec B) = \Frac(B) = \Frac B_S = \Frac \left( (B_S)_0[t,t^{-1} ] \right) = K(\Proj B)^{(1)}$. This proves (a). 

For (b), we need only prove $(\Leftarrow)$ since $(\Rightarrow)$ follows from (a). Let $d = \dim B =  \dim \Spec B$. Suppose $\Spec B$ is rational. Since $K(X) = (B_S)_0$ is a subfield of
$\Frac B \cong \bk(x_1,\dots,x_d)$ and $\trdeg_\bk K(X) = d-1$;
in particular $X$ is unirational. If $d=1$ then $X$ is a point and $K(X) = \bk$. If $d=2$ then Luroth's Theorem gives that $K(X) \isom \bk(t)$ for some $t \in \bk(x_1, x_2)$, so $K(\Proj B) \isom \bk(t) = \bk^{(1)}$ is rational.  If $d=3$ then
$X$ is a unirational surface and so Castelnuovo's
rationality criterion implies that $X$ is rational.
\end{proof}

Recall that by Corollary \ref{ProjIsoCorollary}, for every Pham-Brieskorn ring $B = B_{a_0, \dots, a_n}$, there exists  $B' = B_{b_0, \dots, b_n}$ such that $\Proj B \isom \Proj B'$ and $\cotype(b_0, \dots, b_n) = 0$. Furthermore, the tuple $(b_0, \dots, b_n)$ is easy to determine. So, in order to decide whether $\Proj B_{a_0, \dots, a_n}$ is rational, it suffices to consider the special case where $\cotype(a_0, \dots, a_n) = 0$.

\subsection{The 2-dimensional case} 

We recall the following characterization of $\PPP^1$. 

\begin{lemma}\label{genusZeroCurve}\cite[Example IV.1.3.5]{Hartshorne}
    Let $\bk$ be an algebraically closed field of characteristic zero and let $X$ be a smooth projective curve over $\bk$. Then $X \isom \PPP^1$ if and only if $p_g(X) = 0$. 
\end{lemma}

For $\bk = \bar{\bk}$, determining whether or not $B_{a_0, a_1, a_2}$ is rational is straightforward. Applying Corollary \ref{ProjIsoCorollary}, find $(b_0, b_1, b_2)$ such that $\cotype(b_0, b_1, b_2) = 0$ and $\Proj B_{a_0, a_1, a_2} \isom \Proj B_{b_0, b_1, b_2}$. Let $X = \Proj B_{b_0, b_1, b_2}$. Since $X$ is quasismooth and well-formed it is normal; since $X$ is 1-dimensional it is smooth. By Lemma \ref{FletcherCohomology}, $p_g(X) = \dim_\bk (B_{b_0, b_1, b_2})_\alpha$ where $\alpha$ is the amplitude of $\Proj B_{b_0, b_1, b_2}$. Then 
\begin{align*}
B_{a_0, a_1, a_2} \text{ is rational } & \iff \Spec B_{a_0, a_1, a_2} \text{ is rational } \\
&  \overset{\ref{CastelnuovoProp}}{\iff} \Proj B_{a_0, a_1, a_2} \text{ is rational } \\
&  \overset{\ref{ProjIsoCorollary}}\iff \Proj B_{b_0, b_1, b_2} \text{ is rational } \\
& \overset{\ref{genusZeroCurve}}{\iff} p_g(\Proj B_{b_0, b_1, b_2}) = 0 \\
& \overset{\ref{FletcherCohomology}}{\iff} \dim_\bk (B_{b_0, b_1, b_2})_\alpha = 0 \\
& \iff (B_{b_0, b_1, b_2})_\alpha = 0.
\end{align*}
Not only is the question of whether $B_{a_0, a_1, a_2}$ is rational easily decidable for a particular $(a_0, a_1, a_2)$, the set of tuples such that $B_{a_0, a_1, a_2}$ is rational is described by the following, originally proved for $\bk = \Comp$. The general case follows from Corollary \ref{PBdescent}.  

\begin{proposition}\cite[Proposition 5.5, Pham-Brieskorn case]{arzhantsev2018log}\label{S0char}
    Let $\bk$ be an algebraically closed field of characteristic zero. The following are equivalent:
    \begin{enumerate}[\rm(a)]
        \item $\bk[x_0,x_1,x_2] / \lb x_0^{a_0} + x_1^{a_1} + x_2^{a_2} \rb$ is rational;
        \item Up to a permutation of $(a_0,a_1,a_2)$, one of the following holds:
        \begin{enumerate}[\rm(i)]
            \item $\gcd(a_0a_1, a_2) = 1$
            \item $(a_0,a_1,a_2) = (2a_0',2a_1',2a_2')$ where $a_0',a_1',a_2'$ are pairwise relatively prime.  
        \end{enumerate}
        
    \end{enumerate}
\end{proposition}

\subsection{The $3$-dimensional case}

For $\bk = \bar{\bk}$, determining whether or not the threefold $B_{\bk; a_0, a_1, a_2, a_3}$ is rational is once again feasible although it is harder than in the case of surfaces. Once again, by Corollary \ref{ProjIsoCorollary} together with Proposition \ref{CastelnuovoProp} and Corollary \ref{PBdescent} it suffices to answer the following question:
\begin{question}\label{cotypeZeroSurfaces}
    For which $(a_0, a_1, a_2, a_3)$ satisfying $\cotype(a_0, a_1, a_2, a_3) = 0$ is $\Proj B_{\Comp; a_0, a_1, a_2, a_3}$ rational?
\end{question}

Given a surface $X$, we let $\tilde{X}$ denote a resolution of singularites. The answer to Question \ref{cotypeZeroSurfaces} is obtained in 2 steps. 

\begin{enumerate}
\item Identify the varieties $X = \Proj B_{\Comp; a_0, a_1, a_2, a_3}$ of cotype 0 such that $p_g(\tilde{X}) = 0$.
\item Show that every variety identified in Step (1) is rational. 
\end{enumerate}

Suppose $X = \Proj  B_{\Comp; a_0, a_1, a_2, a_3}$ is such that $\tilde{X}$ satisfies $p_g(\tilde{X}) = \dim H^2(\tilde{X}, \OSheaf_{\tilde{X}}) =  0$. Since $X$ has rational singularities (by \ref{qswciProperties}), we obtain $H^2(\tilde{X}, \OSheaf_{\tilde{X}}) \isom H^2(X, \OSheaf_X) \overset{\ref{FletcherCohomology}}{=} (B_{\Comp; a_0, a_1, a_2, a_3})_\alpha$ where $\alpha$ is the amplitude of $X$. Thus to complete Step (1), we must identify those $(a_0, a_1, a_2, a_3)$ of cotype 0 such that $(B_{\Comp; a_0, a_1, a_2, a_3})_\alpha = 0$. Clearly, if $\alpha < 0$, then $\dim_\Comp (B_{\Comp; a_0, a_1, a_2, a_3})_\alpha = 0$. Contrary to the case of hypersurfaces in standard projective space, it is also possible for $\alpha > 0$ and $(B_{\Comp; a_0,a_1,a_2,a_3})_\alpha = 0$ to hold simultaneously. For example, when $(a_0, a_1, a_2, a_3) = (3,3,7,7)$, we obtain $(B_{\Comp; 3,3,7,7})_\alpha = (B_{\Comp; 3,3,7,7})_1 = 0$.  The difficulty occurs in characterizing such cases. 

Eight pages of elementary number theory (see \cite[Theorem A]{chitayat2025rationality}, but don't read the proof) allow us to complete Step (1). Let us assume without loss of generality that $a_0 \leq a_1 \leq a_2 \leq a_3$, equivalently $w_0 \geq w_1 \geq w_2 \geq w_3$. It can be checked using \cite[Theorem 3.14]{chitayat2025rationality} that
\begin{equation}\label{BalphaChar}
(B_{\Comp; a_0, a_1, a_2, a_3})_\alpha = 0 \iff \alpha < 0 \ \text{ or }\ \text{$a_0 = a_1$, $a_2 = a_3$ and $\gcd(a_0,a_2) = 1$}.
\end{equation}

\begin{proposition}\cite[Lemma 1.3]{alexeev2006pezzo}\label{ldp-rational}
Let $X$ be a normal projective surface over $\Comp$ with at most quotient
singularities such that $-K_X$ is ample. Then $X$ is rational.
\end{proposition}

This proposition shows that the $\alpha < 0$ cases from \eqref{BalphaChar} are rational, as demonstrated by the following immediate example.
\begin{example}
    If $X = \Proj B_{\Comp; a_0, a_1, a_2, a_3}$ has cotype 0 and satisfies $\alpha < 0$, then $-K_X$ is ample and so by Proposition \ref{ldp-rational}, $X$ is rational. 
\end{example}
It turns out the case $a_0 = a_1$, $a_2 = a_3$ and $\gcd(a_0,a_2) = 1$ is handled by the following more general proposition.  

\begin{proposition}\cite[Proposition 3.6]{chitayat2025rationality}\label{rationalAlldim}
    Let $a,c \in \Nat^+$ and consider the graded polynomial ring $R = \bk_{c,\dots, c, a, \dots, a}[x_1, \dots, x_k, y_1, \dots, y_\ell]$ where $\bk$ is any field, $k \geq 1$, $\ell \geq 1$, $\deg(x_i) = c$ and $\deg(y_j) = a$ for all $i,j$. Suppose $f = g(x_1, \dots, x_k) + h(y_1, \dots, y_\ell)$ is irreducible and homogeneous of degree $L = \lcm(a,c)$ with $g(x_1, \dots, x_k) \neq 0$ and $h(y_1, \dots, y_\ell) \neq 0$. Then $\Proj(R / \lb f \rb)$ is rational over $\bk$. 
\end{proposition} 
Putting these results together, we conclude:

\begin{theorem}\label{PBTHM}
    Let $\bk$ be an algebraically closed field of characteristic zero. Suppose $a_0 \leq a_1 \leq a_2 \leq a_3$ and  $\cotype(a_0,a_1,a_2,a_3) = 0$.  Then $\Proj B_{\bk; a_0,a_1,a_2,a_3}$ is rational if and only if one of the following holds:
    \begin{enumerate}[\rm(a)]
        \item $a_0 = a_1$, $a_2 = a_3$ and $\gcd(a_0,a_2) = 1$;
        \item $\frac{1}{a_0} + \frac{1}{a_1} + \frac{1}{a_2} + \frac{1}{a_3} > 1$.
    \end{enumerate}
\end{theorem}
\begin{proof}
    By Corollary \ref{PBdescent}, we may assume $\bk = \Comp$; write $X = \Proj B_{\Comp; a_0,a_1,a_2,a_3}$. By Remark \ref{alphaEquivalence}, condition {\rm(b)} is equivalent to $\alpha < 0$.

    $(\Rightarrow)$ Assume $X$ is rational. Then $p_g(\tilde{X}) = \dim H^2(\tilde{X}, \OSheaf_{\tilde{X}})= \dim H^2(X, \OSheaf_X) = \dim B_\alpha = 0$, so by \eqref{BalphaChar} either {\rm(a)} holds or $\alpha < 0$ and hence {\rm(b)} holds.

    $(\Leftarrow)$ If {\rm(b)} holds then $\alpha < 0$, so $-K_X$ is ample by Theorem \ref{PBCanonical2}; as $X$ has quotient singularities, Proposition \ref{ldp-rational} applies and so $X$ is rational. If {\rm(a)} holds, we can apply Proposition \ref{rationalAlldim}.
\end{proof}

\begin{remark}
    Note that condition (a) in Theorem \ref{PBTHM} implies rationality over any $\bk$, but condition (b) still requires $\bk = \bar\bk$. For example, $\Proj B_{\Reals; 2,2,2,2}$ is not rational. (Apply Lemma \ref{rationalPointsAreDense} to the affine open subset $D_+(x_0)$ for example.) 
\end{remark}

\begin{remark}
    Theorem \ref{PBTHM} shows that over $\bk = \bar{\bk}$, $X = \Proj B_{\bk; a_0, a_1, a_2, a_3}$ is rational if and only if $H^2(X, \OSheaf_X) = 0$. This can now be seen to be a special case of the following recent result of Esser and Li.
\end{remark}

\begin{theorem}\cite[Theorem 1.2]{esser2025weighted} Suppose $X$ is a well-formed quasismooth surface in $\PPP_\Comp(w_0, w_1, w_2, w_3)$.  If $H^2(X, \OSheaf_X) = 0$, then $X$ is rational.
\end{theorem}

\subsection{Higher dimensions}

It is probably quite a difficult problem to determine which (affine or weighted projective) Pham-Brieskorn hypersurfaces are rational given that this type of problem is generally difficult even for hypersurfaces in standard projective space $\PPP^n$. There are however some sporadic results on this question.  Proposition \ref{rationalAlldim} is one such result. Recently, Massarenti showed:

\begin{theorem}\cite[Corollary 3.3]{Massarenti2026}
    Let $ n \geq 1$. The $2n$-dimensional Fermat cubic $V_+(x_0^3 + \dots + x_{2n+1}^3) \subseteq \PPP^{2n+1}$ is rational over any field $\bk$ of characteristic not equal to 3. 
\end{theorem}

It is worth noting that Pham-Brieskorn hypersurfaces play a nice role in giving an easy proof of the following classical proposition.

\begin{proposition}\label{quadricRational}
    Let $\bk$ be an algebraically closed field of characteristic zero, let $n \geq 2$ and assume $Q = V_+(q) \subseteq \PPP^n$ is an irreducible reduced quadric hypersurface. Then $Q$ is rational.
\end{proposition}

\begin{proof}
By \cite[Exercise I.5.12]{Hartshorne}, we may assume $q = x_0^2 + \dots + x_r^2$ where $2 \leq r \leq n$ (since $Q$ is assumed to be irreducible and reduced). Since $\bk$ is algebraically closed, $V_+(q) \isom V_+(x_0x_1 +  x_2^2 + \dots + x_r^2) \subseteq \PPP^n$ and the non-empty open subset $D_+(x_0) \cap V_+(q) \isom \Spec \bk[x_1, \dots, x_n] / \lb x_1 + x_2^2 + \dots + x_r^2 \rb \isom \Aff^{n-1}$. This shows that $V_+(q)$ is birationally equivalent to $\PPP^{n-1}$. 
\end{proof}

\begin{problem}
    Let $X_f \subseteq \PPP(w_0, \dots, w_n)$ be a well-formed quasismooth weighted hypersurface of dimension at least 3 that is also a Fano variety. Describe when $X_f$ is rational. 
\end{problem}

\section{The well-formed Fano threefolds of Pham-Brieskorn type}\label{threefoldList}

\begin{nothing} \label{conditionsPB3}
In order to study the rationality and rigidity of $B_{\bk; a_0, a_1, a_2, a_3, a_4}$ or the rationality and (anti-canonical polar) cylindricity of $\Proj B_{\bk; a_0, a_1, a_2, a_3, a_4}$, we produce the list of Pham-Brieskorn rings $B = B_{\bk;a_0, a_1, a_2, a_3, a_4}$ such that $\Proj B$ is a well-formed quasismooth (possibly singular) Fano threefold. Assume without loss of generality that $a_0 \leq a_1 \leq a_2 \leq a_3 \leq a_4$. By Theorem \ref{PBCanonical2} together with Remark \ref{alphaEquivalence}, if we assume that $\Proj B_{a_0, a_1, a_2, a_3, a_4}$ is well-formed, then it is Fano if and only if $\sum_{i = 0}^4 \frac{1}{a_i} > 1$. Since all of the above questions are easily answered by the aforementioned results when $a_0 = 1$ or $a_0 = a_1 = 2$, we assume henceforth that $a_0 \geq 2$ and $a_1 \geq 3$. Thus we want to classify the tuples $(a_0, a_1, a_2, a_3, a_4)$ subject to the following conditions
\begin{enumerate}[\rm(i)]
\item $a_0 \leq a_1 \leq a_2 \leq a_3 \leq a_4$
\item $a_0 \geq 2$ and $a_1 \geq 3$
\item $\cotype(a_0, a_1, a_2, a_3, a_4) = 0$
\item $\sum_{i = 0}^4 \frac{1}{a_i} > 1$
\end{enumerate}
We will see below that there are 6 infinite families and 232 sporadic cases.

\end{nothing}
\begin{lemma}\label{lem:cot}
The following are equivalent: 
\begin{enumerate}[\rm(i)]
\item $\cotype(a_0, a_1, a_2, a_3, a_4)=0$; 
\item for every prime $p$, the maximum of $v_p(a_0),\dots,v_p(a_4)$ is attained at least twice;
\item $\gcd(w_0, \dots, w_{i-1}, \hat{w}_i, w_{i+1}, \dots,  w_4) = 1$ for all $i$.
\end{enumerate}
\end{lemma}
\begin{proof}
    We leave this to the reader. 
\end{proof}

\begin{nothing} Let $\sigma = \frac{1}{a_0} + \frac{1}{a_1} + \frac{1}{a_2}$. There are two cases to handle in order to produce the complete list of tuples satisfying conditions (i) - (iv) in \ref{conditionsPB3}:
\begin{enumerate}[\rm(1)]
    \item $\sigma \geq 1$,
    \item $\sigma < 1$.
\end{enumerate}
\end{nothing}

\noindent {\bf The $\sigma \geq 1$ case: 6 infinite families.} Assume $\sigma  \geq 1$. Since $a_0 \geq 2, a_1 \geq 3$, $(a_0, a_1, a_2)$ is one of 
\begin{equation}\label{6tuples}(2,3,3),\quad(2,3,4),\quad(2,3,5),\quad(2,3,6),\quad(2,4,4),\quad(3,3,3).
\end{equation}
Fix one of the six triples $T=(a_0,a_1,a_2)$ and let $\Pi(T)$ be the set of primes dividing $a_0a_1a_2$. For $p\in \Pi(T)$ set
\[
  e_p=\max_{0 \leq i \leq 2}v_p(a_i),\qquad n_p=\#\{\, i\in\{0,1,2\} : v_p(a_i)=e_p \,\},
\]
so that repeated entries of $T$ are counted with multiplicity (for instance $n_2=2$ when $T=(2,4,4)$, and $n_3=3$ when $T=(3,3,3)$).
 
\begin{proposition}\label{thm:I}
The following are equivalent:
\begin{enumerate}[\rm(a)]
\item Conditions {\rm(i)-(iv)} of \ref{conditionsPB3} are satisfied and $\sigma \geq 1$; 
\item $T = (a_0, a_1, a_2)$ is one of the tuples in \eqref{6tuples}, $a_2 \leq a_3 \leq a_4$ and
\begin{enumerate}
\item[\rm(1)] $v_p(a_3)=v_p(a_4)$ for every prime $p\notin \Pi(T)$;
\item[\rm(2)] for every $p\in \Pi(T)$:
\[
\begin{array}{ll}
  \text{if } n_p\ge2: & \big( v_p(a_3) \le e_p \text{ and } v_p(a_4)\le e_p\big)\ \text{ or }\ \big(v_p(a_3) = v_p(a_4)>e_p\big);\\[3pt]
  \text{if } n_p=1: & \big(v_p(a_3) =e_p,\ v_p(a_4)\le e_p\big)\ \text{ or }\ \big(v_p(a_4)=e_p,\ v_p(a_3) \le e_p\big)\ \text{ or }\ \big(v_p(a_3) = v_p(a_4)>e_p\big).
\end{array}
\]
\end{enumerate}
\end{enumerate}
\end{proposition}
 
\begin{proof}
This follows from Lemma~\ref{lem:cot}(ii). 
\end{proof}
\begin{nothing}
Proposition~\ref{thm:I} gives the following 6 infinite families, 4 of which can be described relatively nicely:
\[
\begin{array}{l|l}
T & (a_3,a_4)\\\hline
(3,3,3) & (t,t)\ \text{any }t;\quad (t,3t)\ \text{with } 3\nmid t\\[2pt]
(2,4,4) & (t,t)\ \text{any }t;\quad (t,2t)\ \text{with } v_2(t)\le1;\quad (t,4t)\ \text{with } t \text{ odd}\\[2pt]
(2,3,6) & (t,t)\ \text{any }t;\quad (t,2t),\ 2\nmid t;\quad (t,3t),\ 3\nmid t;\quad (t,6t),\ \gcd(t,6)=1; \quad  (2t,3t),\ \gcd(t,6)=1\\[2pt]
(2,3,3) & (t,t),\ 2\mid t;\quad (t,2t),\ 2\nmid t;\quad (t,3t),\ 2\mid t,\ 3\nmid t; \quad (t,6t),\ \gcd(t,6)=1;\quad (2t,3t),\ \gcd(t,6)=1\\[2pt]
(2,3,4) & \text{as in Proposition~\ref{thm:I}}\\[2pt]
(2,3,5) & \text{as in Proposition~\ref{thm:I}}\\[2pt]
\end{array}
\]
\end{nothing}
\noindent {\bf The $\sigma < 1$ case: 232 sporadic cases.} Finally we list the finitely many cases that satisfy conditions (i)-(iv) of \ref{conditionsPB3} and $\sigma<1$. To deduce the finiteness, one can see that conditions (i) and (iv) imply  $a_0 \leq  4$, $a_1 < \frac{4}{1-\frac{1}{a_0}}$, $a_2 < \frac{3}{1 - \frac{1}{a_0}} - \frac{1}{a_1}$, $a_3 <\frac{2}{1-\sigma}$ and condition (iii) implies $a_4 \mid \lcm(a_0, a_1, a_2, a_3)$. 
 
\begin{longtable}{llll}
\caption{The $232$ tuples satisfying conditions (i)-(iv) of \ref{conditionsPB3} and $\sigma<1$.}
\label{tab:sporadic}\\
\hline\endfirsthead
\hline\endhead
\hline\endfoot
(2,3,7,7,42) & (2,3,8,12,24) & (2,4,5,10,20) & (2,6,6,7,21) \\
(2,3,7,8,168) & (2,3,8,13,312) & (2,4,5,11,220) & (2,6,6,8,8) \\
(2,3,7,9,126) & (2,3,8,14,168) & (2,4,5,12,15) & (2,6,6,9,9) \\
(2,3,7,10,105) & (2,3,8,15,40) & (2,4,5,12,30) & (2,6,6,10,10) \\
(2,3,7,10,210) & (2,3,8,15,120) & (2,4,5,12,60) & (2,6,6,11,11) \\
(2,3,7,11,462) & (2,3,8,16,48) & (2,4,5,13,260) & (2,7,7,8,8) \\
(2,3,7,12,28) & (2,3,8,17,408) & (2,4,5,14,140) & (3,3,4,4,4) \\
(2,3,7,12,84) & (2,3,8,18,72) & (2,4,5,15,60) & (3,3,4,4,6) \\
(2,3,7,13,546) & (2,3,8,19,456) & (2,4,5,16,80) & (3,3,4,4,12) \\
(2,3,7,14,21) & (2,3,8,20,120) & (2,4,5,17,340) & (3,3,4,5,20) \\
(2,3,7,14,42) & (2,3,8,21,56) & (2,4,5,18,180) & (3,3,4,5,60) \\
(2,3,7,15,70) & (2,3,8,21,168) & (2,4,5,19,380) & (3,3,4,6,12) \\
(2,3,7,15,210) & (2,3,8,22,264) & (2,4,5,20,20) & (3,3,4,7,28) \\
(2,3,7,16,336) & (2,3,8,23,552) & (2,4,5,28,35) & (3,3,4,7,84) \\
(2,3,7,17,714) & (2,3,8,24,24) & (2,4,6,6,12) & (3,3,4,8,8) \\
(2,3,7,18,63) & (2,3,8,30,40) & (2,4,6,7,84) & (3,3,4,8,24) \\
(2,3,7,18,126) & (2,3,9,9,18) & (2,4,6,8,24) & (3,3,4,9,36) \\
(2,3,7,19,798) & (2,3,9,10,45) & (2,4,6,9,36) & (3,3,4,10,20) \\
(2,3,7,20,420) & (2,3,9,10,90) & (2,4,6,10,60) & (3,3,4,10,60) \\
(2,3,7,21,42) & (2,3,9,11,198) & (2,4,6,11,132) & (3,3,4,11,44) \\
(2,3,7,22,231) & (2,3,9,12,36) & (2,4,6,12,12) & (3,3,4,11,132) \\
(2,3,7,22,462) & (2,3,9,13,234) & (2,4,6,15,20) & (3,3,4,12,12) \\
(2,3,7,23,966) & (2,3,9,14,63) & (2,4,7,7,28) & (3,3,4,13,52) \\
(2,3,7,24,56) & (2,3,9,14,126) & (2,4,7,8,56) & (3,3,4,14,28) \\
(2,3,7,24,168) & (2,3,9,15,90) & (2,4,7,9,252) & (3,3,4,15,20) \\
(2,3,7,25,1050) & (2,3,9,16,144) & (2,4,7,12,21) & (3,3,4,16,16) \\
(2,3,7,26,273) & (2,3,9,17,306) & (2,4,8,8,8) & (3,3,4,20,20) \\
(2,3,7,26,546) & (2,3,9,18,18) & (2,4,9,9,12) & (3,3,5,5,5) \\
(2,3,7,27,378) & (2,3,10,10,15) & (2,5,5,5,10) & (3,3,5,5,15) \\
(2,3,7,28,84) & (2,3,10,10,30) & (2,5,5,6,6) & (3,3,5,6,10) \\
(2,3,7,29,1218) & (2,3,10,11,165) & (2,5,5,6,15) & (3,3,5,6,30) \\
(2,3,7,30,35) & (2,3,10,11,330) & (2,5,5,6,30) & (3,3,5,7,35) \\
(2,3,7,30,70) & (2,3,10,12,20) & (2,5,5,7,14) & (3,3,5,7,105) \\
(2,3,7,30,105) & (2,3,10,12,60) & (2,5,5,7,70) & (3,3,5,8,40) \\
(2,3,7,30,210) & (2,3,10,13,195) & (2,5,5,8,8) & (3,3,5,10,10) \\
(2,3,7,31,1302) & (2,3,10,13,390) & (2,5,5,8,40) & (3,3,6,6,6) \\
(2,3,7,32,672) & (2,3,10,14,105) & (2,5,5,9,18) & (3,3,6,7,14) \\
(2,3,7,33,154) & (2,3,10,14,210) & (2,5,5,9,90) & (3,3,6,8,8) \\
(2,3,7,33,462) & (2,3,10,15,15) & (2,5,5,10,10) & (3,3,6,10,10) \\
(2,3,7,34,357) & (2,3,10,15,30) & (2,5,5,11,22) & (3,3,7,7,7) \\
(2,3,7,34,714) & (2,3,10,18,45) & (2,5,5,12,12) & (3,3,8,8,8) \\
(2,3,7,35,210) & (2,3,10,21,35) & (2,5,5,13,26) & (3,4,4,4,6) \\
(2,3,7,36,252) & (2,3,11,11,66) & (2,5,5,14,14) & (3,4,4,4,12) \\
(2,3,7,37,1554) & (2,3,11,12,44) & (2,5,5,16,16) & (3,4,4,5,15) \\
(2,3,7,38,399) & (2,3,11,12,132) & (2,5,5,18,18) & (3,4,4,5,30) \\
(2,3,7,38,798) & (2,3,11,13,858) & (2,5,6,6,10) & (3,4,4,5,60) \\
(2,3,7,39,182) & (2,3,12,12,12) & (2,5,6,6,15) & (3,4,4,6,6) \\
(2,3,7,39,546) & (2,3,12,13,52) & (2,5,6,6,30) & (3,4,4,6,12) \\
(2,3,7,40,840) & (2,3,12,14,28) & (2,5,6,7,105) & (3,4,4,7,21) \\
(2,3,7,41,1722) & (2,3,12,15,20) & (2,5,6,7,210) & (3,4,4,9,9) \\
(2,3,7,42,42) & (2,3,12,16,16) & (2,5,6,10,15) & (3,4,5,5,12) \\
(2,3,7,48,112) & (2,3,12,20,20) & (2,5,7,7,10) & (3,5,5,6,6) \\
(2,3,7,66,77) & (2,3,14,14,21) & (2,5,7,10,14) & (4,4,4,4,4) \\
(2,3,8,8,12) & (2,4,5,5,20) & (2,5,8,8,10) & (4,4,4,5,5) \\
(2,3,8,8,24) & (2,4,5,6,60) & (2,5,9,9,10) & (4,4,4,5,10) \\
(2,3,8,9,72) & (2,4,5,7,140) & (2,6,6,6,6) & (4,4,4,6,6) \\
(2,3,8,10,120) & (2,4,5,8,40) & (2,6,6,7,7) & (4,4,4,7,7) \\
(2,3,8,11,264) & (2,4,5,9,180) & (2,6,6,7,14) & (4,4,5,5,5) \\
\end{longtable}

\bibliographystyle{alpha}
\bibliography{bibliography}

@article{murthy1969vector,
	title={Vector bundles over affine surfaces birationally equivalent to a ruled surface},
	author={Murthy, M Pavaman},
	journal={Annals of Mathematics},
	pages={242--253},
	year={1969},
	publisher={JSTOR}
}

@article{WROBEL202043,
title = {Divisor class groups of rational trinomial varieties},
journal = {Journal of Algebra},
volume = {542},
pages = {43-64},
year = {2020},
issn = {0021-8693},
doi = {https://doi.org/10.1016/j.jalgebra.2019.09.021},
url = {https://www.sciencedirect.com/science/article/pii/S0021869319305241},
author = {Milena Wrobel}
}

@article{SGA2,
	title={SGA 2},
	author={Grothendieck, Alexander and others},
	journal={S{\'e}minaire de G{\'e}om{\'e}trie Alg{\'e}brique du Bois Marie-1962-Cohomologie locale des faisceaux coh{\'e}rents et th{\'e}oremes de Lefschetz locaux et globaux (North-Holland, Amsterdam)},
	year={1968}
}

@book{lang2012algebra,
  title={Algebra},
  author={Lang, Serge},
  volume={211},
  year={2012},
  publisher={Springer Science \& Business Media}
}

@book{Hartshorne,
	AUTHOR = {Hartshorne, Robin},
	TITLE = {Algebraic Geometry},
	NOTE = {Graduate Texts in Mathematics, No. 52},
	PUBLISHER = {Springer-Verlag},
	ADDRESS = {New York},
	YEAR = {1977},
	PAGES = {xvi+496},
	ISBN = {0-387-90244-9},
	MRCLASS = {14-01},
	MRNUMBER = {MR0463157 (57 \#3116)},
	MRREVIEWER = {Robert Speiser},
}

@article{Chitayat_Daigle_2019,
	author = {Michael Chitayat and Daniel Daigle},
	title = {On the rigidity of certain {P}ham-{B}rieskorn rings},
	journal = {Journal of Algebra},
	volume = {550},
	pages = {290-308},
	year = {2020},
	issn = {0021-8693},
	doi = {https://doi.org/10.1016/j.jalgebra.2020.01.005},
	url = {https://www.sciencedirect.com/science/article/pii/S0021869320300156}
	
}

@book{ishii2014introduction,
	title={Introduction to Singularities},
	author={Ishii, S.},
	isbn={9784431550822},
	url={https://books.google.ca/books?id=SKUijgEACAAJ},
	year={2014},
	publisher={Springer}
}

@inproceedings{DFM2017,
	address = "Tokyo, Japan",
	author = "Daigle, Daniel and Freudenburg, Gene and Moser-Jauslin, Lucy",
	booktitle = "Algebraic Varieties and Automorphism Groups",
	doi = "10.2969/aspm/07510029",
	pages = "29--48",
	publisher = "Mathematical Society of Japan",
	title = "Locally nilpotent derivations of rings graded by an abelian group",
	url = "https://doi.org/10.2969/aspm/07510029",
	year = "2017"
}

@article{KishimotoProkhorovZaidenberg,
	title={${G}_a$-actions on affine cones},
	author={Kishimoto, Takashi and Prokhorov, Yuri and Zaidenberg, Mikhail},
	journal={Transformation Groups},
	year={2012},
	volume={18},
	pages={1137-1153}
}

@book{kollar_mori_1998, place={Cambridge}, 					series={Cambridge Tracts in Mathematics}, 
	title={Birational Geometry of Algebraic Varieties}, DOI={10.1017/CBO9780511662560}, 
	publisher={Cambridge University Press}, 
	author={Kollar, Janos and Mori, Shigefumi}, year={1998}, 
	collection={Cambridge Tracts in Mathematics}
}

@book{przyjalkowski2026weighted,
  title={Weighted complete intersections},
  author={Przyjalkowski, Victor V and Shramov, Constantin},
  year={2026},
  publisher={Walter de Gruyter GmbH \& Co KG}
}

@misc{hajra2026generalizedzariskicancellationbrieskornpham,
      title={Generalized {Z}ariski cancellation for {B}rieskorn--{P}ham varieties}, 
      author={Buddhadev Hajra and Mohit Upmanyu},
      year={2026},
      eprint={2606.26890},
      archivePrefix={arXiv},
      primaryClass={math.AG},
      url={https://arxiv.org/abs/2606.26890}, 
}

@article{arzhantsev2018log,
  title={Log terminal singularities, platonic tuples and iteration of {C}ox rings},
  author={Arzhantsev, Ivan and Braun, Lukas and Hausen, J{\"u}rgen and Wrobel, Milena},
  journal={European Journal of Mathematics},
  volume={4},
  number={1},
  pages={242--312},
  year={2018},
  publisher={Springer}
}

@book{alexeev2006pezzo,
  title={Del Pezzo and K3 surfaces},
  author={Alexeev, V. and Nikulin, V.V.},
  isbn={9784931469341},
  series={G - Reference,Information and Interdisciplinary Subjects Series},
  url={https://books.google.it/books?id=FY8rAAAAYAAJ},
  year={2006},
  publisher={Mathematical Society of Japan}
}

@article{chitayat2025rationality,
  title={Rationality of weighted hypersurfaces of special degree},
  author={Chitayat, Michael},
  journal={Journal of Algebra},
  volume={665},
  pages={7--29},
  year={2025},
  publisher={Elsevier}
}

@article{chitayat2025rigid,
  title={The rigid {P}ham-{B}rieskorn threefolds},
  author={Chitayat, Michael and Dubouloz, Adrien},
  journal={Advances in Mathematics},
  volume={482},
  pages={110640},
  year={2025}
}

@misc{ChitayatThesis, 
    title={Rigidity of {P}ham-{B}rieskorn {T}hreefolds}, 
    author={Chitayat, Michael}, 
    howpublished = "\url{https://ruor.uottawa.ca/handle/10393/44886}",
    note = {University of Ottawa, Theses and Dissertations},
    year={2023}
}

@inbook{iano-fletcher_2000, 
	place={Cambridge}, 
	series={London Mathematical Society Lecture Note Series}, title={Working with weighted complete intersections}, DOI={10.1017/CBO9780511758942.005}, 
	booktitle={Explicit Birational Geometry of 3-folds}, publisher={Cambridge University Press}, 
	author={Iano-Fletcher, A. R.}, 
	year={2000}, 
	pages={101–174}, 
	collection={London Mathematical Society Lecture Note Series}
}

@article{Dimca1986,
	author = {Dimca, Alexandru},
	journal = {Journal für die reine und angewandte Mathematik},
	pages = {184-193},
	title = {Singularities and coverings of weighted complete intersections.},
	url = {http://eudml.org/doc/152819},
	volume = {366},
	year = {1986},
}

@book {MR1251956,
	AUTHOR = {Bruns, Winfried and Herzog, J\"{u}rgen},
	TITLE = {Cohen-{M}acaulay Rings},
	SERIES = {Cambridge Studies in Advanced Mathematics},
	VOLUME = {39},
	PUBLISHER = {Cambridge University Press, Cambridge},
	YEAR = {1993},
	PAGES = {xii+403},
	ISBN = {0-521-41068-1},
	MRCLASS = {13H10 (13-02)},
	MRNUMBER = {1251956},
	MRREVIEWER = {Matthew Miller},
}

@article{LiuSunExtensions,
	author = {Dayan Liu and Xiaosong Sun},
	title = {On the rigidity of some extensions of domains},
	journal = {Michigan Mathematical Journal},
	publisher = {University of Michigan, Department of Mathematics},
	pages = {1 -- 16},
	year = {2022},
	doi = {10.1307/mmj/20205957},
	URL = {https://doi.org/10.1307/mmj/20205957}
}

@article{freudenburg2013,
	author = "Freudenburg, Gene and Moser-Jauslin, Lucy",
	doi = "10.1307/mmj/1370870372",
	journal = "Michigan Mathematical Journal",
	number = "2",
	pages = "227--258",
	publisher = "University of Michigan, Department of Mathematics",
	title = "Locally nilpotent derivations of rings with roots adjoined",
	url = "https://doi.org/10.1307/mmj/1370870372",
	volume = "62",
	year = "2013"
}

@article{affineCones,
	title     = "AFFINE CONES OVER SMOOTH CUBIC SURFACES",
	author    = "Ivan Cheltsov and Jihun Park and Joonyeong Won",
	year      = "2016",
	doi       = "10.4171/JEMS/622",
	language  = "English",
	volume    = "18",
	pages     = "1537--1564",
	journal   = "Journal of the European Mathematical Society",
	issn      = "1435-9855",
	publisher = "European Mathematical Society Publishing House",
	number    = "7",
}

@article{CylindersInDelPezzoSurfaces,
	author = {Cheltsov, Ivan and Park, Jihun and Won, Joonyeong},
	title = {Cylinders in del {P}ezzo Surfaces},
	journal = {International Mathematics Research Notices},
	volume = {2017},
	number = {4},
	pages = {1179-1230},
	year = {2016},
	issn = {1073-7928},
	doi = {10.1093/imrn/rnw063},
	url = {https://doi.org/10.1093/imrn/rnw063},
	eprint = {https://academic.oup.com/imrn/article-pdf/2017/4/1179/10756009/rnw063.pdf},
}

@book{freudenburg2017algebraic,
	
	title={Algebraic Theory of Locally Nilpotent Derivations},
	author={Freudenburg, Gene},
	isbn={9783662553503},
	series={Encyclopaedia of Mathematical Sciences},
	year={2017},
	publisher={Springer Berlin Heidelberg}
}

@InProceedings{dolgachev,
	author="Dolgachev, Igor",
	editor="Carrell, James B.",
	title="Weighted projective varieties",
	booktitle="Group Actions and Vector Fields",
	year="1982",
	publisher="Springer Berlin Heidelberg",
	address="Berlin, Heidelberg",
	pages="34--71",
	isbn="978-3-540-39528-7"
}

@Article{Cheltsov2010,
	author="Cheltsov, Ivan
	and Park, Jihun
	and Shramov, Constantin",
	title="Exceptional del {P}ezzo Hypersurfaces",
	journal="Journal of Geometric Analysis",
	year="2010",
	volume="20",
	number="4",
	pages="787--816",
}

@article{cheltsov_park_won_2016, 
	title={Cylinders in singular del {P}ezzo surfaces}, 
	volume={152}, 
	DOI={10.1112/S0010437X16007284}, 
	number={6}, 
	journal={Compositio Mathematica}, 
	publisher={London Mathematical Society}, 
	author={Cheltsov, Ivan and Park, Jihun and Won, Joonyeong}, 
	year={2016}, 
	pages={1198–1224}
}

@article {Kali-Zaid_2000,
    AUTHOR = {Kaliman, Shulim and Zaidenberg, Mikhail},
     TITLE = {Miyanishi's characterization of the affine 3-space does not
              hold in higher dimensions},
   JOURNAL = {Annales de l'Institut Fourier (Grenoble)},
    VOLUME = {50},
      YEAR = {2000},
    NUMBER = {6},
     PAGES = {1649--1669},
}

@article{Orlik1977AlgebraicSW,
  title={Algebraic surfaces with k*-action},
  author={Peter Orlik and Philip Wagreich},
  journal={Acta Mathematica},
  year={1977},
  volume={138},
  pages={43-81},
  url={https://api.semanticscholar.org/CorpusID:122368936}
}

@article{orlik1979,
	author = "Orlik, Peter",
	fjournal = "Bulletin (New Series) of the American Mathematical Society",
	journal = "Bulletin of the  American Mathematical Society (N.S.)",
	number = "5",
	pages = "703--720",
	publisher = "American Mathematical Society",
	title = "Singularities and group actions",
	url = "https://projecteuclid.org:443/euclid.bams/1183544714",
	volume = "1",
	year = "1979"
}

@misc{stacks-project,
	shorthand    = {Stacks},
	author       = {The {Stacks Project Authors}},
	title        = {\textit{Stacks Project}},
	howpublished = {\url{https://stacks.math.columbia.edu}},
	year         = {2022},
}

@article{beltramettiRobbiano,
	author = {Beltrametti, Mauro and Robbiano, Lorenzo},
	year = {1986},
	title = {Introduction to the theory of weighted projective spaces},
	volume = {4},
	journal = {Expositiones Mathematicae}
}

@article{Massarenti2026,
author = {Massarenti, Alex},
title = {Rational points on even-dimensional {F}ermat cubics},
journal = {Transactions of the London Mathematical Society},
volume = {13},
number = {1},
pages = {e70028},
doi = {https://doi.org/10.1112/tlm3.70028},
url = {https://londmathsoc.onlinelibrary.wiley.com/doi/abs/10.1112/tlm3.70028},
year = {2026}
}

@article{daigle2004locallyDanielewski,
  title={Locally nilpotent derivations and {D}anielewski surfaces},
  author={Daigle, Daniel},
  year={2004},
  journal = {Osaka Journal of Mathematics}
}

@article{russell1970forms,
  title={Forms of the affine line and its additive group},
  author={Russell, Peter},
  journal={Pacific journal of mathematics},
  volume={32},
  number={2},
  pages={527--539},
  year={1970},
  publisher={Mathematical Sciences Publishers}
}

@book{gortz2020algebraic,
  title={Algebraic geometry I: schemes},
  author={G{\"o}rtz, Ulrich and Wedhorn, Torsten},
  year={2020},
  publisher={Springer}
}

@book{samuel1964ufd,
  author    = {Samuel, Pierre},
  title     = {Lectures on Unique Factorization Domains},
  series    = {Tata Institute of Fundamental Research Lectures on Mathematics},
  number    = {30},
  publisher = {Tata Institute of Fundamental Research},
  address   = {Bombay},
  year      = {1964},
  note      = {Notes by M. Pavman Murthy},
}

@article{cheltsov2020cylinders,
	title={Cylinders in {F}ano varieties}, 
	author={Ivan Cheltsov and Jihun Park and Yuri Prokhorov and Mikhail Zaidenberg},
	year={2021},
	journal={EMS Surveys in Mathematical Sciences},
	volume={8},
	number={1},
	pages={39-105},
	URL={https://ems.press/journals/emss/articles/2504061}
	
}

@article{storch1984picard,
	title={Die {P}icard-{Z}ahlen der {S}ingularit{\"a}ten $t_1^{r_1} + t_2^{r_2} + t_3^{r_3} + t_4^{r_4} = 0$},
	author={Storch, Uwe},
	journal={Journal f{\"u}r die reine und angewandte Mathematik},
	volume={350},
	pages={188--202},
	year={1984}
}

@article{chitayat2026isomorphism,
  title={The Isomorphism Classes of the Surfaces $ x_1^{a_1}+ x_2^{a_2}+ x_3^{a_3}+ 1= 0$},
  author={Chitayat, Michael and Hajra, Buddhadev},
  journal={arXiv preprint arXiv:2605.07617},
  year={2026}
}

@article{gurjar2025classification,
  title={Classification of certain affine plane curves},
  author={Gurjar, S.R. and Pokale, P.},
  journal={Journal of the Ramanujan Mathematical Society},
  volume={40},
  number={1},
  year={2025},
  publisher={RAMANUJAN MATHEMATICAL SOC C/O MNG ED, UNIV MYSORE, DEPT MATHEMATICS, MYSORE~…}
}

@article{esser2025weighted,
  title={Weighted surfaces with maximal {P}icard number},
  author={Esser, Louis and Li, Jennifer},
  journal={arXiv preprint arXiv:2506.14037},
  year={2025}
}

\end{document}